\documentclass[nointlimits]{amsart}

\usepackage[T1]{fontenc} 
\usepackage[utf8]{inputenc} 
\usepackage[USenglish]{babel} 

\usepackage{lmodern}
\usepackage{mathrsfs}
\usepackage{microtype}

\usepackage{url}
\usepackage{natbib}

\usepackage{amsmath, amssymb, amsthm, enumerate, mathscinet}

\usepackage{enumitem}

\makeatletter
\let\save@mathaccent\mathaccent
\newcommand*\if@single[3]{%
  \setbox0\hbox{${\mathaccent"0362{#1}}^H$}%
  \setbox2\hbox{${\mathaccent"0362{\kern0pt#1}}^H$}%
  \ifdim\ht0=\ht2 #3\else #2\fi
  }
\newcommand*\rel@kern[1]{\kern#1\dimexpr\macc@kerna}
\newcommand*\widebar[1]{\@ifnextchar^{{\wide@bar{#1}{0}}}{\wide@bar{#1}{1}}}
\newcommand*\wide@bar[2]{\if@single{#1}{\wide@bar@{#1}{#2}{1}}{\wide@bar@{#1}{#2}{2}}}
\newcommand*\wide@bar@[3]{%
  \begingroup
  \def\mathaccent##1##2{%
    \let\mathaccent\save@mathaccent
    \if#32 \let\macc@nucleus\first@char \fi
    \setbox\z@\hbox{$\macc@style{\macc@nucleus}_{}$}%
    \setbox\tw@\hbox{$\macc@style{\macc@nucleus}{}_{}$}%
    \dimen@\wd\tw@
    \advance\dimen@-\wd\z@
    \divide\dimen@ 3
    \@tempdima\wd\tw@
    \advance\@tempdima-\scriptspace
    \divide\@tempdima 10
    \advance\dimen@-\@tempdima
    \ifdim\dimen@>\z@ \dimen@0pt\fi
    \rel@kern{0.6}\kern-\dimen@
    \if#31
      \overline{\rel@kern{-0.6}\kern\dimen@\macc@nucleus\rel@kern{0.4}\kern\dimen@}%
      \advance\dimen@0.4\dimexpr\macc@kerna
      \let\final@kern#2%
      \ifdim\dimen@<\z@ \let\final@kern1\fi
      \if\final@kern1 \kern-\dimen@\fi
    \else
      \overline{\rel@kern{-0.6}\kern\dimen@#1}%
    \fi
  }%
  \macc@depth\@ne
  \let\math@bgroup\@empty \let\math@egroup\macc@set@skewchar
  \mathsurround\z@ \frozen@everymath{\mathgroup\macc@group\relax}%
  \macc@set@skewchar\relax
  \let\mathaccentV\macc@nested@a
  \if#31
    \macc@nested@a\relax111{#1}%
  \else
    \def\gobble@till@marker##1\endmarker{}%
    \futurelet\first@char\gobble@till@marker#1\endmarker
    \ifcat\noexpand\first@char A\else
      \def\first@char{}%
    \fi
    \macc@nested@a\relax111{\first@char}%
  \fi
  \endgroup
}
\makeatother

\newcommand{\R}{\mathbb{R}}
\newcommand{\rn}{\R^n}
\newcommand{\N}{\mathbb{N}}
\newcommand{\M}{\mathscr{M}}
\newcommand{\Mpl}{\M^+}
\newcommand{\F}{\mathcal{F}}
\newcommand{\I}{\mathscr{S}}
\DeclareMathOperator{\Log}{Log}

\newcommand{\optTar}[1]{Y_{#1}}

\newcommand{\VmXOm}{V^m_0X(\Omega)}
\newcommand{\VmLamOm}{V^m_0\Lambda_w^q(\Omega)}

\newcommand{\fundX}[1][X]{\varphi_{#1}}
\newcommand{\fundOptX}[1][X]{\varphi_{\optTar{#1}}}
\newcommand{\fundOptXAsoc}[1][X]{\varphi_{\optTar{#1}'}}

\newcommand{\RR}{R}

\newcommand*\dd{\mathop{}\!\mathrm{d}}

\DeclareMathOperator*{\esssup}{ess\,sup}

\numberwithin{equation}{section}

\theoremstyle{plain}
\newtheorem{theorem}{Theorem}[section]

\newtheorem{proposition}{Proposition}[section]

\theoremstyle{definition}
\newtheorem{definition}{Definition}[section]
\newtheorem{remark}{Remark}[section]

\makeatletter
\let\c@corollary=\c@theorem
\let\c@proposition=\c@theorem
\let\c@remark=\c@theorem
\let\c@definition=\c@theorem
\makeatother

\usepackage{hyperref}
\newcommand{\TITLE}{Non-strict Singularity of Optimal Sobolev Embeddings}
\newcommand{\AUTHORS}{Jan Lang and Zdeněk Mihula}
\hypersetup{
	unicode=true,
	breaklinks=true,
	pdftitle={\TITLE},
	pdfauthor={\AUTHORS},
	}

\usepackage[capitalise]{cleveref} 

\title{\TITLE}
\author{\AUTHORS}
\address{Jan Lang, Department of Mathematics, The Ohio State University, 231 West 18th Avenue, Columbus, OH 43210-1174, USA; Czech Technical University in Prague, Faculty of Electrical Engineering, Department of Mathematics, Technick\'a~2, 166~27 Praha~6, Czech Republic}
\email{lang.162@osu.edu}
\urladdr{\href{https://orcid.org/0000-0003-1582-7273}{0000-0003-1582-7273}}
\address{Zden\v ek Mihula, Czech Technical University in Prague, Faculty of Electrical Engineering, Department of Mathematics, Technick\'a~2, 166~27 Praha~6, Czech Republic}
\email{mihulzde@fel.cvut.cz}
\urladdr{\href{https://orcid.org/0000-0001-6962-7635}{0000-0001-6962-7635}}

\begin{document}
\setcitestyle{numbers}
\bibliographystyle{plainnat}

\subjclass[2020]{46E35, 46E30, 47B60} 
\keywords{strictly singular operators, Sobolev spaces, compactness, optimal spaces, rearrangement-invariant spaces, Lambda spaces}
\thanks{This research was partly supported by grant no.~23-04720S of the Czech Science Foundation.}

\begin{abstract}
We investigate the operator-theoretic property of strict singularity for optimal Sobolev embeddings within the general framework of rearrangement-invariant function spaces (r.i.\ spaces). 

More specifically, we focus on studying the ``quality'' of non-compactness for optimal Sobolev embeddings 
	$V^m_0X(\Omega)\to Y_X(\Omega)$, where $X$ is a given r.i.\ space and $Y_X$  is the corresponding optimal target r.i.\ space (i.e., the smallest among all r.i.\ spaces).
	 
For the class of sub-limiting norms (i.e., the norms whose fundamental function satisfies $\varphi_{Y_X}(t)\approx t^{-m/n}\varphi_X(t)$ as $t\to0^+$), we construct suitable spike-function sequences that establish a general framework for proving non-strict singularity of optimal (and thus non-compact) sublimiting Sobolev embeddings.

As an application, we show that optimal sublimiting Sobolev embeddings are not strictly singular in a rather large subclass of r.i.\ spaces, namely weighted Lambda spaces $X=\Lambda^q_w$, $q\in[1, \infty)$. Except for the endpoint case $X=L^{n/m,1}$, our spike-function construction enables us to construct a subspace of $V^m_0X$ that is isomorphic to $\ell_q$, which we then leverage to prove the non-strict singularity of the corresponding optimal Sobolev embedding.
\end{abstract}

\maketitle

\section{Introduction}
Although particular cases were already known at the beginning of the 20th century, the systematic study of Sobolev embeddings date back to the 1930s and the pioneering work of Sergei Sobolev \cite{Sob38}, who established that weak derivatives control integral norms of functions and thereby guarantee their certain continuity and integrability properties. Subsequent refinements include those by Gagliardo~\cite{G:58} and Nirenberg~\cite{N:59} in the 1950s, who independently introduced important interpolation inequalities and extended the Sobolev embedding established by Sobolev for $p>1$ to the endpoint case $p=1$. A classical form of these Sobolev embeddings, for a bounded domain $\Omega\subseteq\rn$, $1\leq m < n$, and $1\leq p<n/m$, reads as:
\begin{equation}\label{intro:classical_Sob_emb}
W^{m,p}_0(\Omega) = W^m_0 L^p(\Omega) \hookrightarrow L^{p^*}(\Omega),
\quad\frac1{p^*}=\frac1p-\frac m n.
\end{equation}
The so-called Sobolev exponent $p^*$ is optimal in the sense that it cannot be replaced by any bigger exponent $q>p^*$, and the embedding \eqref{intro:classical_Sob_emb} is not compact (see \cite{AFbook}).

Already in the 1960s--1970s, it was apparent that the Lebesgue scale of function spaces is not sufficient for capturing various subtle limiting or borderline behaviors, and more general, finer-grained scales of function spaces, such as Lorentz or Orlicz spaces, became widely used (e.g., see~\cite{BW:80,D:71,H:79,P:65,S:81,S:72b,T:67,Y:61}). A large number of such function spaces are particular instances of so-called rearrangement-invariant function spaces (r.i.\ spaces, see~\cite{BS}). Among others, landmark contributions by Maz'ya and Talenti elucidated fine properties of solutions to elliptic PDEs via rearrangements and capacitary estimates (see~\cite{M:73, Mabook, T:76, T:76b}), which fortified the importance of studying the structure of general Sobolev-type embeddings and various symmetrization techniques (see also~\cite{B:19} and references therein).

More recently, a great deal of effort has been devoted to describing optimal spaces in various Sobolev embeddings in the general framework of r.i.\ spaces. In particular, Kerman and Pick in \cite{KP:06} (see also~\cite{EKP:00}) completely characterized optimal target r.i.\ spaces in Sobolev embeddings on bounded regular domains in~$\rn$. An equivalent form of the Sobolev embeddings in question is
\begin{equation}\label{intro:general_Sob_emb_in_ri_sps}
V_0^m X(\Omega) \hookrightarrow Y(\Omega),
\end{equation}
where both $X$ and $Y$ are r.i.\ spaces, $\Omega\subseteq\rn$ is a bounded domain in $\rn$, and $V_0^m X(\Omega)$ is a homogeneous $m$th-order Sobolev space built upon $X$. They completely characterized the optimal target r.i.\ space $Y(\Omega)$ in \eqref{intro:general_Sob_emb_in_ri_sps}. We will denote this optimal space by $Y_X$. The optimality is meant in the sense that $Y_X$ the smallest r.i.\ space $Y$ with which \eqref{intro:general_Sob_emb_in_ri_sps} is valid for a given r.i.\ space $X$. Note that \eqref{intro:classical_Sob_emb} is a particular case of \eqref{intro:general_Sob_emb_in_ri_sps} with $Y(\Omega)=L^{p^*}(\Omega)$ and $W^{m,p}_0(\Omega) = V_0^m L^p(\Omega)$, up to equivalent norms (see Subsection~\ref{subsec:Sob_spaces} for more information and precise definitions), and the optimal target r.i.\ space $Y_X(\Omega)$ in \eqref{intro:classical_Sob_emb} is the Lorentz space $L^{p^*,p}(\Omega)\subsetneq L^{p^*}(\Omega)$ (see~\cite{On:63, P:66, T:98}). Furthermore, in the subsequent paper \cite{KP:08}, they also completely characterized the compactness of \eqref{intro:general_Sob_emb_in_ri_sps}. In particular, it turns out that the optimal Sobolev embedding $V_0^m X(\Omega) \hookrightarrow Y_X(\Omega)$ is never compact (see also~\cite{CM:19,CM:22,S:12,S:15}). The interested reader is referred to \cite{BC:21, CPS:15, CPS:20} and references therein for more information on optimal Sobolev embeddings in various settings.

Despite the extensive work on the quality of compactness for the classical Sobolev spaces, where the quality is usually studied via the speed of decay for different $s$-numbers or entropy numbers (see~\cite{ET:96, P:85} and references therein), there was, until recently, only a limited amount of work devoted to the study of the ``quality'' of non-compactness of non-compact Sobolev-type embeddings. For non-compact Sobolev embeddings, it was observed that standard measures of non-compactness (related to Kolmogorov $s$-numbers or entropy numbers) are not satisfactory (see \cite{CS:90,LMOP:21, LMP:24FirstView}) as they often fail to distinguish between different target spaces. For example, from this point of view, the ``quality'' of non-compactness of the optimal Sobolev embedding $V_0^{m,p}(\Omega) \hookrightarrow L^{p^*,p}(\Omega)$ is the same as that of the non-compact, non-optimal Sobolev embeddings $V_0^{m,p}(\Omega) \hookrightarrow L^{p^*,q}(\Omega) \supsetneq L^{p^*,p}(\Omega)$ for every $q\in(p,\infty]$. This motivates the use of different operator-theoretic concepts, such as \emph{strict singularity} or \emph{finite strict singularity} (see~Subsection~\ref{4.3}), which form the main focus of this paper.

  Compact embeddings, which are often approximable by finite-rank operators, are always finitely strictly singular and then also strictly singular; however, optimal Sobolev embeddings into r.i.\ spaces are never compact. Moreover, it was showed that the optimal Sobolev embedding $V^m_0 L^p(\Omega)\hookrightarrow L^{p^*,p}(\Omega)$ is not strictly singular (see \cite{LM:23})\textemdash unlike the non-optimal Sobolev embedding \eqref{intro:classical_Sob_emb}, which is finitely strictly singular, despite not being compact. The non-strict singularity of the optimal Sobolev embedding was later proved also in more general settings (see~\cite{CLY:25preprint,GLM:25}). To the best of our knowledge, the only known exception to this phenomenon is the optimal limiting Sobolev embedding $V_0^m L^{n/m,1}(\Omega)\hookrightarrow \mathcal C_b(\Omega)$ into the space of bounded continuous functions (see~\cite{S:81}), which is not only strictly singular but also finitely strictly singular (see \cite{LM:19}), despite being non-compact.
	
	This raises several natural questions: under what conditions on the underlying norms do non-compact optimal Sobolev embeddings become non-strictly singular?  Moreover, is this phenomenon pervasive across the broader landscape of r.i.\ spaces?
	
To date, in all known examples, proofs and techniques used to establish that optimal non-compact Sobolev embeddings fail to be strictly singular rest on two key ingredients. They together ensure that it is rather simple to find a suitable infinite-dimensional subspace of the Sobolev-type space in question that attests the non-strict singularity of the corresponding optimal Sobolev embedding. First, they involve classical $L^p$ or $\ell_p$ structures (e.g., $V_0^{m,p}$ and $L^{p^*,p}$, $V_0^m L^{p, q}$ and $L^{p^*,q}$, or $B^s_{p,q}$ spaces, see~\cite{T:83}). Second, they exhibit a certain dilation invariance. For example, given a function $u\in V_0^{m,p}(\Omega)$, the transformation $u \mapsto u_\kappa \in V_0^{m,p}(\Omega)$, $\kappa\geq1$, where $u_\kappa(x) = \kappa^{n/p-m}u(\kappa x)$, produces a family of functions that concentrate (i.e., their supports vanish) as $\kappa\to\infty$, but both $\|u_\kappa\|_{V_0^{m,p}(\Omega)} = \|u\|_{V_0^{m,p}(\Omega)}$ and $\|u_\kappa\|_{L^{p^*,p}(\Omega)} = \|u\|_{L^{p^*,p}(\Omega)}$ remain unchanged. However, such an invariance is rarely at our disposal when one ventures beyond these rather simple, well-behaved cases into the realm of general r.i.\ spaces.

  In this paper, we develop a general framework for addressing these questions, extending the theory well beyond the $L^p$ setting.
After recalling the necessary background on r.i.\ spaces, optimal Sobolev embeddings, and strictly singular operators in Section \ref{Preliminaries}, we introduce in Section \ref{OptimalSublim} the class of \emph{sublimiting} r.i.~norms (\cref{def:sublimiting}, see also~\cref{prop:sublimiting_X_neces_conds}). These norms are characterized by the following relation between their fundamental function $\varphi_X$ (i.e., the norm of a characteristic function of a set with prescribed measure) and the fundamental function $\varphi_{Y_X}$ of the optimal norm in the corresponding Sobolev embedding:
\begin{equation*}
\varphi_{Y_X}(t)\approx t^{-m/n}\varphi_X(t)\quad(t\to0^+).
\end{equation*}
For such norms, we prove a key test-function construction (\cref{thm:small_support_big_norm}) showing that given an optimal sublimiting Sobolev embedding $V_0^mX(\Omega) \to Y_X(\Omega)$, one can find sequences of radially concentrated “spike” functions whose norms under the Sobolev embedding do not collapse. Notably, this construction bypasses the fact that we do not have any dilation invariance at our disposal for a general optimal embedding $V_0^mX(\Omega) \to Y_X(\Omega)$. Moreover, the optimal target space $Y_X$ can be replaced by any r.i.\ space $Y$ having the same fundamental function as $Y_X$.

As our main application (Section \ref{Section4}), we consider the family of classical Lorentz spaces of type Lambda (Lambda spaces, for short)
$\Lambda^q_w$, $q\in[1, \infty)$, defined by
\begin{equation*}
\|f\|_{\Lambda^q_w}
=\Biggl( \int_0^1 f^*(t)^q w(t) \dd{t} \Biggr)^\frac1{q},
\end{equation*}
where $w$ is a suitable weight. This class of function spaces encompasses a large number of r.i.\ spaces, such as Lebesgue spaces, Lorentz spaces, some Orlicz spaces, or Lorentz--Zygmund spaces (see~Subsection~\ref{subsec:riSps}).  We first show that these norms are sublimiting under natural conditions on the weight $w$ (see~\cref{prop:when_is_optimal_Lambda_nonlimiting}). Then we show that the corresponding optimal Sobolev embedding
\begin{equation}\label{intro:optimal_Sobolev_embedding_Lambda}
\VmLamOm = V_0^m X(\Omega) \hookrightarrow Y_X(\Omega) = \Lambda^q_{w_{opt}}(\Omega)
\end{equation}
fails to be strictly singular (\cref{thm:optimal_embedding_Lambda_spaces_not_SS}), where $w_{opt}$ is an optimal weight (see the proof of \cref{prop:optimal_for_Lambda_is_Lambda} for its definition). To achieve this, we leverage our key test-function construction and exploit the known fact that Lambda spaces with exponent $q$ contain an order isomorphic copy of $\ell_q$ (see~\cite[Theorem~1]{KM:04}). It should be noted that this known fact still has to be substantially modified and refined to ensure the synergy and coherence of arguments, leading to the desired result. Ultimately, we are able to obtain a rather explicit infinite-dimensional subspace of $\VmLamOm$, isomorphic to $\ell_q$, on which the Sobolev embedding \eqref{intro:optimal_Sobolev_embedding_Lambda} is an isomorphism onto its image. This shows that the phenomenon observed before only in very special cases of optimal Sobolev embeddings (by using their particular special forms and structures) is prevalent in a large number of general optimal Sobolev embeddings, which can be described as~\eqref{intro:optimal_Sobolev_embedding_Lambda}.

\section{Preliminaries} \label{Preliminaries}

Throughout the entire paper, we assume that $\Omega\subseteq\rn$, $n\geq2$, is a bounded domain such that $|\Omega| = 1$. We also use the convention that $0\cdot\infty=1/\infty=0$. We will often write $A\lesssim B$, where $A$ and $B$ are nonnegative expressions, if there is a positive constant $c$ such that $A\leq c\cdot B$. The constant $c$ may depend on some parameters in the expression. If not stated explicitly,  what the constant may depend on and what it may not should be obvious from the context. We also write $A\gtrsim B$ with the obvious meaning, and $A\approx B$ when $A\lesssim B$ and $A\gtrsim B$ simultaneously.

\subsection{Rearrangement-invariant function spaces}\label{subsec:riSps}

In this subsection, $(\RR,\mu)$ is a nonatomic measure space with $\mu(\RR) = 1$. In the rest of this paper, $(\RR,\mu)$ will be either $\Omega$ or the interval $(0,1)$ (both equipped with the Lebesgue measure). The set of all $\mu$-measurable functions on $(\RR, \mu)$ is denoted by $\M(\RR, \mu)$. We also denote by $\Mpl(\RR, \mu)$ the set of all those functions from $\M(\RR, \mu)$ that are nonnegative $\mu$-a.e.\ in $\RR$. 

The \textit{nonincreasing rearrangement} $f^* \colon  (0,\infty) \to [0, \infty ]$ of a function $f\in \M(\RR, \mu)$ is
defined as
\begin{equation*}
f^*(t)=\inf\{\lambda\in(0,\infty) : \mu(\{x\in\RR: |f(x)| > \lambda\})\leq t\},\ t\in(0,\infty).
\end{equation*}
The operation $f\mapsto f^*$ is not subadditive in general, but we have
\begin{equation}\label{prel:pointwise_inequality_sum_of_rearrangements}
(f+g)^*(t) \leq f^*(t/2) + g^*(t/2) \quad \text{for every $t\in(0, \infty)$}
\end{equation}
and all $f,g\in\M(\RR, \mu)$, for which the sum $f+g$ is well defined. We say that two functions $f\in \M(\RR,\mu)$ and $g\in \M(S,\nu)$, where $(S,\nu)$ is a possibly different measure space, are \emph{equimeasurable} if $f^* = g^*$. In particular, $f$ and $f^*$ are equimeasurable. The function $f^*$ is nonincreasing and right continuous. We have $\esssup_{x\in\RR}|f(x)| = f^*(0^+)$ and $f^*(t) = 0$ for every $t\geq \mu(\{x\in\RR: f(x) \neq 0\})$. If $f=\chi_E$ for some $\mu$-measurable $E\subseteq\RR$, then $f^* = \chi_{(0, \mu(E))}$. If $|f|\leq |g|$ $\mu$-a.e.\ in $\RR$, then $f^*\leq g^*$. We also define the \textit{maximal function} $f^{**} \colon (0,\infty) \to [0, \infty ]$ of $f\in \M(R,\mu)$ as
\begin{equation*}
f^{**}(t)=\frac1{t} \int_0^ t f^{*}(s) \dd s,\ t\in(0,\infty).
\end{equation*}
We always have $f^*\leq f^{**}$. Unlike $f^*$, $f^{**}$ need not be equimeasurable with $f$.

We say that a functional $\|\cdot\|_{X(0,1)}\colon \Mpl(0,1) \to [0, \infty]$ is a \emph{rearrangement-invariant function norm} (we will write an \emph{r.i.\ function norm}) if for all $f,g,f_k\in\Mpl(\RR, \mu)$, $k\in\N$, and $\lambda\geq0$:
\begin{enumerate}[label=(P\arabic*), ref=(P\arabic*)]
	\item $\|f + g\|_{X(0,1)} \leq \|f\|_{X(0,1)} + \|g\|_{X(0,1)}$; $\|f\|_{X(0,1)} = 0$ if and only if $f=0$~a.e.; $\|\lambda f\|_{X(0,1)} = \lambda \|f\|_{X(0,1)}$;
	\item $\|f\|_{X(0,1)} \leq \|g\|_{X(0,1)}$ whenever $f\leq g$ a.e.;
	\item $\|f_k\|_{X(0,1)} \nearrow \|f\|_{X(0,1)}$ whenever $f_k\nearrow f$ pointwise a.e.;
	\item $\|1\|_{X(0,1)} = \|\chi_{(0,1)}\|_{X(0,1)} < \infty$;
	\item $\|f\|_{L^1(0,1)}\leq C_{X} \|f\|_{X(0,1)}$, where $C_{X}\in(0, \infty)$ is a constant that is independent of $f$ (but it may depend on $\|\cdot\|_{X(0,1)}$);
	\item $\|f\|_{X(0,1)} = \|g\|_{X(0,1)}$ whenever $f$ and $g$ are equimeasurable.
\end{enumerate}

Given an r.i.\ function norm $\|\cdot\|_{X(0,1)}$, we extend it to all functions $f\in\M(0,1)$ by
\begin{equation*}
\|f\|_{X(0,1)} = \|\,|f|\,\|_{X(0,1)},\ f\in\M(0,1).
\end{equation*}

The \emph{rearrangement-invariant space} (\emph{r.i.\ space} for short) $X(\RR, \mu)$ is defined as
\begin{equation*}
X(\RR, \mu) = \{f\in\M(\RR,\mu): \|f\|_{X(\RR, \mu)} < \infty\},
\end{equation*}
where the functional $\|\cdot\|_{X(\RR, \mu)}$ is defined as
\begin{equation*}
\|\cdot\|_{X(\RR, \mu)} = \|f^*\|_{X(0,1)},\ f\in\M(\RR, \mu).
\end{equation*}
It is a Banach space (in particular, $\|\cdot\|_{X(\RR, \mu)}$ is a norm on $X(\RR, \mu)$), and the functions from $X(\RR, \mu)$ are finite $\mu$-a.e. An r.i.\ space always contains simple functions. By a simple function, we mean a (finite) linear combination of characteristic functions of sets of finite measure.

Prototypical examples of r.i.\ spaces are the usual $L^p(\RR, \mu)$ spaces (i.e., \emph{Lebesgue spaces}), $p\in[1, \infty]$. The rearrangement invariance of the usual Lebesgue norm $\|\cdot\|_{L^p(\RR, \mu)}$ follows from the layer cake formula (e.g., \cite[Chapter~2, Proposition~1.8]{BS} or \cite[Theorem~1.13]{LL:01}), which tells us that
\begin{equation}\label{prel:Lp_norm_with_rearrangement}
\|f\|_{L^p(\RR,\mu)} = \|f^*\|_{L^p(0, 1)} \quad \text{for every $f\in\M(\RR, \mu)$}.
\end{equation}
We will introduce other r.i.\ spaces later.

We say that a Banach space $A$ is \emph{embedded in} a Banach space $B$, and write $A\hookrightarrow B$, if $A\subseteq B$ (in the set-theoretic sense) and the inclusion is continuous\textemdash that is, there is a constant $C>0$ such that $\|f\|_B \leq C \|f\|_A$ for every $f\in A$. We say that $A$ and $B$ are equivalent, and write $A = B$, if $A\hookrightarrow B$ and $B\hookrightarrow A$. In other words, $A$ and $B$ coincide in the set-theoretic sense and their norms are equivalent. If $X(\RR, \mu)$ and $Y(\RR, \mu)$ are r.i.\ spaces, then $X(\RR, \mu) \hookrightarrow Y(\RR, \mu)$ if and only if $X(\RR, \mu) \subseteq Y(\RR, \mu)$. Note that $X(\RR, \mu) \hookrightarrow Y(\RR, \mu)$ if and only if $X(0,1) \hookrightarrow Y(0,1)$. Furthermore, we always have
\begin{equation}\label{prel:Linfty_and_Lone_emb}
L^\infty(\RR,\mu) \hookrightarrow X(\RR,\mu) \hookrightarrow L^1(\RR,\mu).
\end{equation}

With any r.i.\ function norm $\|\cdot\|_{X(0,1)}$, there is associated another r.i.\ function norm, $\|\cdot\|_{X'(0,1)}$, defined for $g \in  \M(0,1)$ as
\begin{equation*}
\|g\|_{X'(0,1)}=\sup_{\|f\|_{X(0,1)}\leq1} \int_0^1 |f(t)||g(t)| \dd{t},\ g\in  \M(0,1).
\end{equation*}
The r.i.\ function norm $\|\cdot\|_{X'(0,1)}$ is called the \emph{associate norm} of $\|\cdot\|_{X(0,1)}$. The resulting r.i.\ space $X'(\RR, \mu)$ is called the \emph{associate space} (of $X(\RR, \mu)$). We have
\begin{equation}\label{prel:holder}
\int_{\RR} |f||g| \dd\mu \leq \|f\|_{X(\RR, \mu)} \|g\|_{X'(\RR,\mu)} \quad \text{for all $f,g\in\M(\RR,\mu)$}.
\end{equation}
We always have $X''(\RR, \mu) = X(\RR, \mu)$ with equality of the norms, where $X''(\RR, \mu) = (X')'(\RR, \mu)$. Furthermore, we have
$X(\RR, \mu) \hookrightarrow Y(\RR, \mu)$ if and only if $Y'(\RR, \mu) \hookrightarrow X'(\RR, \mu)$. For example, if $X(\RR, \mu) = L^p(\RR, \mu)$, where $p\in[1, \infty]$, then $X'(\RR, \mu) = L^{p'}(\RR, \mu)$, with equality of the norms.

\emph{Dilation} is always bounded on r.i.\ spaces over $(0,1)$. More precisely, for every $a>0$, the dilation operator $D_a\colon \M(0,1) \to \M(0,1)$ is define as
\begin{equation*}
D_af(t) = f(t/a)\chi_{(0,a)}(t),\ t\in(0,1),\ f\in\M(0,1).
\end{equation*}
We have
\begin{equation}\label{prel:dilation:bounded}
\|Df\|_{X(0,1)} \leq \max\{1,a\} \|f\|_{X(0,1)} \quad \text{for every $f\in\M(0,1)$}
\end{equation}
and every r.i.\ space $X(0,1)$.

The \emph{fundamental function} of an r.i.\ space $X(\RR, \mu)$ is the function $\fundX\colon [0, 1] \to [0, \infty)$ defined as
\begin{equation*}
\fundX(t) = \|\chi_E\|_{X(\RR, \mu)},\ t\in[0, 1],
\end{equation*}
where $E\subseteq\RR$ is any $\mu$-measurable set with $\mu(E) = t$. The rearrangement invariance of $\|\cdot\|_{X(\RR, \mu)}$ ensures that the fundamental function is well defined. Note that $\fundX[X(\RR,\mu)] = \fundX[X(0,1)]$. Furthermore, $\fundX$ is quasiconcave and satisfies
\begin{equation*}
\fundX(t) \fundX[X'](t) = t \quad \text{for every $t\in[0,1]$}.
\end{equation*}
A function $\varphi\colon [0,1] \to [0, \infty)$ is \emph{quasiconcave} if $\varphi$ is nondecreasing, the function $(0,1]\ni t\mapsto \varphi(t)/t$ is nonincreasing, and $\varphi(t) = 0$ if and only if $t = 0$. A quasiconcave function is always continuous on $(0, 1]$, but it may have a jump at $0$. For example, if $X(\RR, \mu) = L^p(\RR, \mu)$, then $\fundX(t) = t^{1/p}$, $t\in(0,1]$. The fundamental function of an r.i.\ space $X(\RR, \mu)$ need not be concave, but $X(\RR, \mu)$ can always be equivalently renormed with another r.i.\ function norm in such a way that the resulting fundamental function is concave. More specifically, we can do it in such a way that the resulting fundamental function is the least nondecreasing concave majorant of $\fundX$. Given a quasiconcave function $\varphi\colon[0,1]\to[0,\infty)$, its least nondecreasing concave majorant $\tilde{\varphi}$ satisfies
\begin{equation}\label{prel:concave_majorant_pointwise_equiv}
\varphi(t) \leq \tilde{\varphi}(t) \leq 2 \varphi(t) \quad \text{for every $t\in[0,1]$}.
\end{equation}
Since we do not aim to compute exact values of constants, we may always without loss of generality assume that the fundamental function of an r.i.\ space is concave.

Given a quasiconcave function $\varphi\colon[0,1]\to[0,\infty)$, we define functionals $\|\cdot\|_{M_\varphi(0, 1)}$ and $\|\cdot\|_{m_\varphi(0, 1)}$ as
\begin{align*}
\|f\|_{M_\varphi(0, 1)} &= \sup_{t\in(0,1)} f^{**}(t) \varphi(t),\ f\in\M(0,1),\\
\|f\|_{m_\varphi(0, 1)} &= \sup_{t\in(0,1)} f^*(t) \varphi(t),\ f\in\M(0,1).
\end{align*}
The functional $\|\cdot\|_{M_\varphi(0, 1)}$ is an r.i.\ function norm and we have $\fundX[M_\varphi] = \varphi$. Furthermore, we always have
\begin{equation}\label{prel:fundamental_endpoint_embedding_into_M}
X(\RR, \mu) \hookrightarrow M_{\varphi_X}(\RR, \mu).
\end{equation}
The functional $\|\cdot\|_{m_\varphi(0, 1)}$ need not be an r.i.\ function norm, but we always have $M_{\varphi}(\RR, \mu) \hookrightarrow m_{\varphi}(\RR, \mu)$, with $m_\varphi(\RR, \mu)$ being defined in the obvious way. Furthermore, we have
\begin{equation*}
M_\varphi(\RR, \mu) = m_\varphi(\RR, \mu)
\end{equation*}
if and only if (e.g., see \cite[Lemma~2.1]{MO:19} and references therein)
\begin{equation}\label{prel:M=m_int_cond}
\frac1{t}\int_0^t \frac1{\varphi(s)} \dd{s} \lesssim \frac1{\varphi(t)} \quad \text{for every $t\in(0,1)$}.
\end{equation}
When $\|\cdot\|_{X(0,1)}$ is an r.i.\ function norm, we set $\|\cdot\|_{M_X(0,1)} = \|\cdot\|_{M_{\varphi_X}(0, 1)}$ and $\|\cdot\|_{m_X(0,1)} = \|\cdot\|_{m_{\varphi_X}(0, 1)}$. The spaces $M_X(\RR, \mu)$ and $m_X(\RR, \mu)$ are sometimes called \emph{Marcinkiewicz endpoint spaces}.

Let $\varphi\colon[0,1]\to[0,\infty)$ be a not identically zero concave function such that $\varphi(0) = 0$. We define the functional $\|\cdot\|_{\Lambda_\varphi(0, 1)}$ as
\begin{equation*}
\|f\|_{\Lambda_\varphi(0, 1)} = \|f\|_{L^\infty(0,1)}\varphi(0^+) + \int_0^1 f^*(s) \varphi'(s) \dd{s},\ f\in\M(0,1).
\end{equation*}
The functional $\|\cdot\|_{\Lambda_\varphi(0, 1)}$ is an r.i.\ function norm and we have $\fundX[\Lambda_\varphi] = \varphi$. Furthermore, we always have
\begin{equation}\label{prel:fundamental_endpoint_embedding_LorLam_to_X}
\Lambda_{\varphi_X}(\RR,\mu) \hookrightarrow X(\RR,\mu),
\end{equation}
assuming $\fundX$ is concave. Under the same assumption, we set $\|\cdot\|_{\Lambda_X(0, 1)} = \|\cdot\|_{\Lambda_{\varphi_X}(0, 1)}$. The space $\Lambda_X(\RR, \mu)$ is sometimes called a \emph{Lorentz endpoint space}.

The associate spaces of $M_\varphi$ and $\Lambda_\varphi$ satisfy the following. Given a quasiconcave function $\varphi\colon[0,1]\to[0,\infty)$, we have
\begin{equation}\label{prel:asoc_of_M_is_endpoint_Lorentz}
\big(M_\varphi\big)'(\RR, \mu) = \Lambda_{\tilde{\psi}}(\RR, \mu),
\end{equation}
where $\tilde{\psi}\colon[0,1]\to[0,\infty)$ is the least nondecreasing concave majorant of the quasiconcave function $\psi(t) = t/\varphi(t)$, $t\in(0,1]$. On the other hand, given a not identically zero concave function $\varphi\colon[0,1]\to[0,\infty)$ such that $\varphi(0) = 0$, we have
\begin{equation}\label{prel:asoc_of_endpoint_Lorentz_is_M}
\Big(\Lambda_\varphi\Big)'(0,1) = M_\psi(0,1),
\end{equation}
where $\psi(t)$ = $t/\varphi(t)$, $t\in(0,1]$.

The interested reader is referred to \cite{BS} for more information about r.i.\ spaces.

A large number of r.i.\ spaces are particular instances of so-called \emph{Lambda spaces} $\Lambda^q_w$ for a suitable weight $w$. We say that $w\in\Mpl(0,1)$ is a \emph{weight} if $0 < W(t) \leq W(1) < \infty$ for every $t\in(0,1)$, where
\begin{equation*}
W(t) = \int_0^t w(s) \dd{s},\ t\in[0,1].
\end{equation*}
For $q\in[1, \infty)$ and a weight $w$, we defined the functional $\|\cdot\|_{\Lambda^q_w(0,1)}$ as
\begin{equation*}
\|f\|_{\Lambda^q_w(0,1)} = \Big( \int_0^1 f^*(t)^q w(t) \dd{t} \Big)^\frac1{q},\ f\in\M(0,1).
\end{equation*}
Given a weight $w$, the functional $\|\cdot\|_{\Lambda^q_w(0,1)}$ is equivalent to an r.i.\ function norm if and only if (for $q\in(1,\infty)$, see \cite{GS:14,S:90}; for $q=1$, see \cite{CGS:96})
\begin{equation}\label{prel:when_Lambda_equiv_ri}
\begin{cases}
W(t_2)/t_2 \lesssim W(t_1)/t_1 \quad \text{for all $0<t_1 \leq t_2 \leq 1$}, \quad&\text{if $q=1$},\\
\int_0^t W(s)^{1 - q'} s^{q' - 1} \dd{s} \lesssim t^{q'} W(t)^{1 - q'} \quad \text{for every $t\in(0,1)$}, \quad&\text{if $q\in(1, \infty)$}.
\end{cases}
\end{equation}
When this is the case, we will consider $\|\cdot\|_{\Lambda^q_w(0,1)}$ to be an r.i.\ function norm. Furthermore, the conditions \eqref{prel:when_Lambda_equiv_ri} imply that
\begin{equation}\label{prel:Lambda_delta2}
\sup_{t\in(0,1/2)} \frac{W(2t)}{W(t)} < \infty.
\end{equation}
For example, for $p,q\in[1, \infty)$, we define the weight $w_{p,q}$ as
\begin{equation}\label{prel:Lorentz_Lpq_weight}
w_{p,q}(t) = t^{\frac{q}{p} - 1},\ t\in(0,1).
\end{equation}
Then \eqref{prel:when_Lambda_equiv_ri} is satisfied if and only if $p\in(1,\infty)$ or $p=q=1$, and
\begin{equation*}
\Lambda^q_{w_{p,q}}(\RR,\mu) = L^{p,q}(\RR, \mu)
\end{equation*}
is the \emph{Lorentz space} $L^{p,q}$. In particular, when $p=q\in[1, \infty)$, $\Lambda^p_{w_{p,p}}(\RR,\mu) = L^{p,p}(\RR, \mu) = L^p(\RR,\mu)$ is the Lebesgue space (recall \eqref{prel:Lp_norm_with_rearrangement}). Another (more general) important example is
\begin{equation*}
w_{p,q,\alpha}(t) = t^{\frac{q}{p}-1}\log(e/t)^{q\alpha},\ t\in(0,1),
\end{equation*}
for $p\in[1, \infty]$, $q\in[1, \infty)$, and $\alpha\in\R$. Then \eqref{prel:when_Lambda_equiv_ri} is satisfied if and only if $p=q=1$ and $\alpha\geq0$, $p\in(1, \infty)$, or $p=\infty$ and $\alpha + 1/q<0$, and
\begin{equation*}
\Lambda^q_{w_{p,q,\alpha}}(\RR,\mu) = L^{p,q,\alpha}(\RR,\mu)
\end{equation*}
is the \emph{Lorentz--Zygmund space} $L^{p,q,\alpha}$. In particular, for $p=1$ and $\alpha\geq0$ or $p\in(1, \infty)$,
\begin{equation*}
\Lambda^p_{w_{p,p,\alpha/p}}(\RR,\mu) = L^p(\Log L)^\alpha(\RR, \mu)
\end{equation*}
is the Orlicz space $L^A(\RR,\mu) = L^p(\Log L)^\alpha(\RR, \mu)$ generated by a Young function $A$ satisfying
\begin{equation*}
A(t) \approx t^p (1+\log t)^\alpha \quad \text{near $\infty$}.
\end{equation*}
See \cite{BR:80, OP:99} for more information. Recall that a function $A\colon[0, \infty] \to [0, \infty]$ is called a \emph{Young function} if it is convex, left-continuous, vanishing at $0$, and not constant on the entire interval $(0, \infty)$. Given a Young function $A$, the \emph{Orlicz space} $L^A(\RR, \mu)$ is defined as the collection of those $f\in\M(\RR, \mu)$ for which
\begin{equation*}
\|f\|_{L^A(\RR,\mu)} = \inf\Bigg\{ \lambda>0: \int_{\RR} A\Big( \frac{|f(x)|}{\lambda} \Big) \dd\mu(x) \leq 1 \Bigg\} < \infty.
\end{equation*}
We have $\|f\|_{L^A(\RR, \mu)} = \|f^*\|_{L^A(0,1)}$ for every $f\in\M(\RR, \mu)$, and $\|\cdot\|_{L^A(0,1)}$ is an r.i.\ function norm.
\subsection{Sobolev spaces built upon r.i.\ spaces}\label{subsec:Sob_spaces}
Let $X(\Omega)$ be an r.i.\ space and $m\in\N$. The \emph{Sobolev space} $\VmXOm$ is defined as the space of all $m$-times weakly differentiable functions $u$ on $\Omega$ whose continuation by $0$ outside $\Omega$ is an $m$-times weakly differentiable function and 
\begin{equation*}
\|u\|_{\VmXOm} = \|\nabla^m u\|_{X(\Omega)} = \|\,|\nabla^m u|\,\|_{X(\Omega)} < \infty,
\end{equation*}
where $\nabla^m u$ is the (arbitrarily arranged) vector of all $m$th order (weak) derivatives of $u$. It is a Banach space (in particular, $\|\cdot\|_{\VmXOm}$ is a norm on $\VmXOm$). In fact, we have (see~\cite[Proposition~4.5]{CPS:15})
\begin{equation*}
\|u\|_{W_0^m X(\Omega)} = \|u\|_{X(\Omega)} + \sum_{j = 1}^m \| \nabla^j u\|_{X(\Omega)} \approx \|\nabla^m u\|_{X(\Omega)} \quad \text{for every $u\in \VmXOm$}.
\end{equation*}
In particular, we could equivalently work with the Sobolev space $W_0^m X(\Omega)$ instead of $\VmXOm$.

By \cite[Theorem~A]{KP:06}, for every r.i.\ space $X(\Omega)$, there is an r.i.\ space $Y_X(\Omega)$ that is optimal in Sobolev embeddings for $\VmXOm$ (equivalently, for $W_0^m X(\Omega)$) in the following sense: The Sobolev embedding
\begin{equation}\label{prel:optimal_Sob_emb}
\VmXOm \hookrightarrow Y_X(\Omega)
\end{equation}
is valid, and whenever \eqref{prel:optimal_Sob_emb} is valid with $Y_X(\Omega)$ replaced by an r.i.\ space $Y(\Omega)$, we have $Y_X(\Omega) \hookrightarrow Y(\Omega)$. Furthermore, we have
\begin{equation}\label{prel:optimal_asoc_norm}
\|g\|_{Y_X'(0,1)} = \|t^\frac{m}{n} g^{**}(t)\|_{X'(0,1)} \quad \text{for every $g\in\M(0,1)$}.
\end{equation}

\subsection{Strictly singular operators} \label{4.3}
Let $A$ and $B$ be Banach spaces ($A$ being infinite-dimensional), and $T\colon A \to B$ be a bounded linear operator. We say that $T$ is \emph{strictly singular} if it is not bounded from below on any infinite-dimensional subspace of $A$. More precisely, for every infinite-dimensional linear subspace $Z\subseteq A$, we have
\begin{equation*}
    \inf_{\substack{x\in Z\\\|x\|_A = 1}} \|Tx\|_B = 0.
\end{equation*}
The operator $T$ as said to be \emph{finitely strictly singular} if for every $\varepsilon>0$ there exists $n\in \N$ such that for each $Z\subseteq A$, $\dim(Z)\ge n$, we have $\inf_{x\in Z, |x\|_A = 1} \|Tx\|_B \leq \varepsilon$.
We always have
\begin{equation*}
    \emph{$T$ is compact} \Rightarrow \emph{$T$ is finitely strict.\ sing.} \Rightarrow \emph{$T$ is strictly singular} \Rightarrow \emph{$T$ is bounded},
\end{equation*}
and none of the implications can be reversed in general. The interested reader is referred to \cite{AK:16,LR-P:14,P:04} and references therein for more information about strictly singular operators and related concepts.

\section{Optimal Sublimiting Embeddings} \label{OptimalSublim}
The main result of this section is \cref{thm:small_support_big_norm}, which will play an important role in the proof of our main result, \cref{thm:optimal_embedding_Lambda_spaces_not_SS}. To prove \cref{thm:small_support_big_norm}, we first need to establish a few results regarding the optimal Sobolev embedding \eqref{prel:optimal_Sob_emb}, which are of independent interest. We start with a proposition telling us that the ``small'' and ``capital'' endpoint spaces $m$ and $M$ corresponding to optimal spaces in Sobolev embeddings \eqref{prel:optimal_Sob_emb} coincide.
\begin{proposition}
Let $\|\cdot\|_{X(0,1)}$ be an r.i.\ function norm. Then
\begin{equation}\label{E:m_and_M_coincide_for_optimal_tar:1}
M_{Y_X}(0,1) = m_{Y_X}(0,1),
\end{equation}
where $Y_X$ is the optimal target r.i.\ space in the Sobolev embedding \eqref{prel:optimal_Sob_emb}.
\end{proposition}
\begin{proof}
Assume that $m\geq n$. Since $V^m_0 L^1(\Omega) \to L^\infty(\Omega)$ when $m\geq n$, it follows from \eqref{prel:Linfty_and_Lone_emb} that we have $\optTar{X}(0,1) = L^\infty(0,1)$ for every r.i.\ space $X(0,1)$. Hence, \eqref{E:m_and_M_coincide_for_optimal_tar:1} trivially holds. Therefore, from now on, we assume that $m < n$.

Recall that the validity of \eqref{E:m_and_M_coincide_for_optimal_tar:1} is equivalent to that of \eqref{prel:M=m_int_cond} with $\varphi = \fundOptX$. It follows from \cite[Lemma~2.1]{S:72} that in order to establish \eqref{prel:M=m_int_cond} with $\varphi = \fundOptX$, it is sufficient to show that
\begin{equation}\label{lem:m_and_M_coincide_for_optimal_tar:eq:2}
\fundOptXAsoc(a) \lesssim \Big( \frac{a}{b} \Big)^\frac{m}{n} \fundOptXAsoc(b) \quad \text{for all $0 < a\leq b < 1$}.
\end{equation}
To this end, using \eqref{prel:optimal_asoc_norm}, we have
\begin{equation}\label{lem:m_and_M_coincide_for_optimal_tar:eq:1}
\fundOptXAsoc(a) = \Big\| t^{\frac{m}{n}-1} \int_0^t \chi_{(0,a)}(s) \dd{s} \Big\|_{X'(0,1)} \quad \text{for every $a\in(0,1)$}.
\end{equation}
Note that
\begin{align}
\Big\| t^{\frac{m}{n}-1} \int_0^t \chi_{(0,a)}(s) \dd{s} \Big\|_{X'(0,1)} &= \frac{a}{b} \Big\| t^{\frac{m}{n}-1} \int_0^{\frac{b}{a}t} \chi_{(0,b)}(\tau) \dd{\tau} \Big\|_{X'(0,1)} \nonumber\\
&= \Big(\frac{a}{b}\Big)^\frac{m}{n} \Big\| \Big(\frac{b}{a}t \Big)^{\frac{m}{n} - 1} \int_0^{\frac{b}{a}t} \chi_{(0,b)}(\tau) \dd{\tau} \Big\|_{X'(0,1)} \nonumber\\
&\leq \Big(\frac{a}{b}\Big)^\frac{m}{n} \Big\| \Big(\frac{b}{a}t \Big)^{\frac{m}{n} - 1} \int_0^{\frac{b}{a}t} \chi_{(0,b)}(\tau) \dd{\tau} \chi_{(0, a/b)}(t)\Big\|_{X'(0,1)} \nonumber\\
&\quad+ \Big(\frac{a}{b}\Big)^\frac{m}{n} \Big\| \Big(\frac{b}{a}t \Big)^{\frac{m}{n} - 1} \int_0^{\frac{b}{a}t} \chi_{(0,b)}(\tau) \dd{\tau} \chi_{(a/b,1)}(t) \Big\|_{X'(0,1)}. \label{lem:m_and_M_coincide_for_optimal_tar:eq:3}
\end{align}
We have
\begin{align}
	\Big\| \Big(\frac{b}{a}t \Big)^{\frac{m}{n} - 1} \int_0^{\frac{b}{a}t} \chi_{(0,b)}(\tau) \dd{\tau} \chi_{(0, a/b)}(t)\Big\|_{X'(0,1)} &= \| D_{a/b}[s^\frac{m}{n}\chi_{(0,b)}^{**}(s)](t) \|_{X'(0,1)} \nonumber\\
	&\leq \|t^\frac{m}{n}\chi_{(0,b)}^{**}(t)\|_{X'(0,1)} = \fundOptXAsoc(b) \label{lem:m_and_M_coincide_for_optimal_tar:eq:4}
\end{align}
thanks to \eqref{prel:dilation:bounded}. Furthermore, using \eqref{prel:holder}, we see that
\begin{align}
\Big\| \Big(\frac{b}{a}t \Big)^{\frac{m}{n} - 1} \int_0^{\frac{b}{a}t} \chi_{(0,b)}(\tau) \dd{\tau} \chi_{(a/b,1)}(t) \Big\|_{X'(0,1)} &= \int_0^1 \chi_{(0,b)}(\tau) \dd{\tau} \Big\| \Big(\frac{b}{a}t \Big)^{\frac{m}{n} - 1}  \chi_{(a/b,1)}(t) \Big\|_{X'(0,1)} \nonumber\\
&\leq \fundOptXAsoc(b) \fundOptX(1) \Big\| \Big(\frac{b}{a}t \Big)^{\frac{m}{n} - 1}  \chi_{(a/b,1)}(t) \Big\|_{X'(0,1)} \nonumber\\
&\leq \fundOptXAsoc(b) \fundOptX(1)  \fundX[X'](1). \label{lem:m_and_M_coincide_for_optimal_tar:eq:5}
\end{align}
Finally, by combining \eqref{lem:m_and_M_coincide_for_optimal_tar:eq:1}, \eqref{lem:m_and_M_coincide_for_optimal_tar:eq:3}, \eqref{lem:m_and_M_coincide_for_optimal_tar:eq:4}, and \eqref{lem:m_and_M_coincide_for_optimal_tar:eq:5}, we obtain \eqref{lem:m_and_M_coincide_for_optimal_tar:eq:2}, which concludes the proof.
\end{proof}

\begin{definition}\label{def:sublimiting}
Let $1\leq m < n$. Let $\|\cdot\|_{X(0,1)}$ be an r.i.\ function norm. We say that $\|\cdot\|_{X(0,1)}$ is \emph{sublimiting} (in Sobolev embeddings for $\VmXOm$) if
\begin{equation}\label{E:sublimiting}
\fundX[\optTar{X}](t) \approx t^{-\frac{m}{n}} \fundX(t) \quad \text{for every $t\in(0,1)$},
\end{equation}
where $\optTar{X}$ is the optimal target r.i.\ space in the Sobolev embedding \eqref{prel:optimal_Sob_emb}.
\end{definition}

\begin{remark}\label{rem:sublimiting_norms}\
\begin{enumerate}[label=(\roman*), ref=(\roman*)]
	\item Since the functions on both sides of \eqref{E:sublimiting} are positive and continuous on the interval $[\delta, 1]$ for each $\delta\in(0,1)$, in order to establish the equivalence, it is sufficient to establish it for every $t\in(0, \delta)$ for some $\delta>0$.
	\item\label{rem:sublimiting_norms:item:nonlimiting_int_cond} By \cite[Theorem~4.5]{MPT:23}, an r.i.\ function norm $\|\cdot\|_{X(0,1)}$ is sublimiting if
		\begin{equation}\label{E:not_limiting_int_cond}
			\int_t^1 \frac{s^{-1 + \frac{m}{n}}}{\fundX(s)} \dd{s} \lesssim \frac{t^\frac{m}{n}}{\fundX(t)} \quad \text{for every $t\in(0,1)$}
		\end{equation}
		is satisfied. In particular, a usually easily verifiable sufficient condition for \eqref{E:not_limiting_int_cond} (and in turn also for \eqref{E:sublimiting}), which can often be used, is that there is $\alpha > m/n$ such that the function
		\begin{equation*}
			(0,1)\ni t\mapsto \frac{t^\alpha}{\fundX(t)}
		\end{equation*}
		is equivalent to a nonincreasing function.
\end{enumerate}

\end{remark}

It is not entirely accurate to call r.i.\ function norms $\|\cdot\|_{X(0,1)}$ for which \eqref{E:sublimiting} is satisfied sublimiting, because it is also satisfied in the limiting situation $X = L^{\frac{n}{m},1}$. However, the following proposition shows that $X = L^{\frac{n}{m},1}$ is the only limiting case, in the sense that $\VmXOm$ is embedded in $L^\infty(\Omega)$, for which \eqref{E:sublimiting} is satisfied. 
\begin{proposition}\label{prop:sublimiting_X_neces_conds}
Let $1\leq m < n$. Let $\|\cdot\|_{X(0,1)}$ be an r.i.\ function norm. If $\|\cdot\|_{X(0,1)}$ is sublimiting, then either
\begin{equation}\label{prop:sublimiting_X_neces_conds:not_in_Lnm1}
X(0,1) \not\subseteq L^{\frac{n}{m},1}(0,1)
\end{equation}
or
\begin{equation}\label{prop:sublimiting_X_neces_conds:equiv_Lnm1}
X(0,1) = L^{\frac{n}{m},1}(0,1).
\end{equation}
\end{proposition}
\begin{proof}
The conditions \eqref{prop:sublimiting_X_neces_conds:not_in_Lnm1} and \eqref{prop:sublimiting_X_neces_conds:equiv_Lnm1} are clearly mutually exclusive. Assume that $\|\cdot\|_{X(0,1)}$ is sublimiting and that \eqref{prop:sublimiting_X_neces_conds:not_in_Lnm1} is not true, that is, we have
\begin{equation}\label{prop:sublimiting_X_neces_conds:eq:1}
X(0,1) \hookrightarrow  L^{\frac{n}{m},1}(0,1).
\end{equation}
We need to show that \eqref{prop:sublimiting_X_neces_conds:equiv_Lnm1} is true. In view of \eqref{prop:sublimiting_X_neces_conds:eq:1}, we only need to prove that
\begin{equation}\label{prop:sublimiting_X_neces_conds:eq:2}
L^{\frac{n}{m},1}(0,1) \hookrightarrow X(0,1).
\end{equation}
To this end, by \eqref{prel:Linfty_and_Lone_emb} and \cite[Theorem~A]{KP:06} (cf.~\cite{S:81}), the validity of \eqref{prop:sublimiting_X_neces_conds:eq:1} is equivalent to the fact that
\begin{equation*}
Y_X(0,1) = L^\infty(0,1),
\end{equation*}
which in turn is equivalent to the fact that (see~\cite[Theorem~5.2]{S:12}, cf.~\cite[Chapter~2, Theorem~5.5]{BS})
\begin{equation*}
\lim_{t\to0^+}\fundOptX(t) > 0.
\end{equation*}
Combining this and \eqref{E:sublimiting}, we see that
\begin{equation*}
t^{-\frac{m}{n}} \fundX(t) \approx 1 \quad \text{for every $t\in(0,1)$}.
\end{equation*}
It follows that
\begin{equation}\label{prop:sublimiting_X_neces_conds:eq:3}
\fundX(t) \approx t^{\frac{m}{n}} \quad \text{for every $t\in(0,1)$}.
\end{equation}
Now, note that
\begin{equation}\label{prop:sublimiting_X_neces_conds:eq:4}
L^{\frac{n}{m},1}(0,1) = \Lambda_\varphi(0,1) \qquad \text{for $\varphi(t) = t^\frac{n}{m}$, $t\in(0,1)$.}
\end{equation}
Hence, by combining \eqref{prop:sublimiting_X_neces_conds:eq:4}, \eqref{prop:sublimiting_X_neces_conds:eq:3}, and \eqref{prel:fundamental_endpoint_embedding_LorLam_to_X}, we obtain \eqref{prop:sublimiting_X_neces_conds:eq:2}.
\end{proof}

\begin{remark}\label{rem:borderling_limiting_cases}
There are borderline cases that are not sublimiting in the sense of \cref{def:sublimiting} and at the same time, nor are they (super)limiting in the sense that $\VmXOm$ is embedded in $L^\infty(\Omega)$. In particular, the latter means that \eqref{prop:sublimiting_X_neces_conds:not_in_Lnm1} is satisfied for such spaces. Typical examples of such borderline cases are Lorentz spaces $X=L^{\frac{n}{m}, q}$ with $q\in(1, \infty]$. Then $Y_X(\Omega)$ is the Lorentz--Zygmund space $L^{\infty,q,-1}(\Omega)$ (e.g., see~\cite[Theorem~5.1]{CP:16}), whose fundamental function satisfies
\begin{equation*}
\fundOptX(t) \approx \log(e/t)^{-1+\frac1{q}} \quad \text{for every $t\in(0, 1)$}.
\end{equation*}
Consequently, \eqref{E:sublimiting} reads as
\begin{equation*}
\log(e/t)^{-1+\frac1{q}} \approx t^{-\frac{m}{n}}t^{\frac{m}{n}} \quad \text{for every $t\in(0, 1)$}
\end{equation*}
in these cases, which is clearly not true.
\end{remark}

Our next proposition is a technical result, which will come handy in the proof of \cref{thm:small_support_big_norm}. We first need to make some definitions and recall an important result. For $a\in(0,1]$, we define
\begin{equation*}
\F_a = \Big\{ f = \sum_{j = 1}^M \alpha_j \chi_{(0, b_j)}: M\in\N, \alpha_j >0, 0<b_1<\dots<b_M\leq a \Big\}
\end{equation*}
and set $\F = \F_1$. In other words, $\F_a$ is the set of all positive nonincreasing simple functions on $(0,1)$ whose support is inside the interval $[0,a]$. Furthermore, for $0\leq a < b \leq 1$, we define
\begin{equation*}
\I_{a,b} = \{f=\alpha\chi_{(a,b)}: \alpha>0\}.
\end{equation*}
Clearly
\begin{equation*}
\I_{0,b} \subseteq \F_b \subseteq \F \quad \text{for every $0< b\leq 1$}.
\end{equation*}

By \cite[Lemma~4.9]{EMMP:20}, we have
\begin{equation}\label{E:H_on_simple}
\Big\| \int_t^1 f(s) s^{-1 + \frac{m}{n}} \dd{s} \Big\|_{Z(0,1)} \approx \Big\| \sum_{j = 1}^M \alpha_j b_j^{\frac{m}{n}} \chi_{(0, b_j)} \Big\|_{Z(0,1)} \quad \text{for every $f\in\F$}
\end{equation}
and for every r.i.\ function norm $\|\cdot\|_{Z(0,1)}$, in which the multiplicative constants depend only on $m$ and $n$.

\begin{proposition}\label{prop:cutting_of_support}
Let $1\leq m < n$. Let $\|\cdot\|_{X(0,1)}$ be an r.i.\ function norm. Assume that $\fundX$ is concave. If $\|\cdot\|_{X(0,1)}$ is sublimiting, then there is a constant $C$ such that
\begin{align}
\sup_{\substack{\|h\|_{\Lambda_X(0,1)}\leq1 \\ h\in\I_{a/2,a}}} \Big\| \int_t^1 h(s) s^{-1 + \frac{m}{n}} \dd{s} \Big\|_{M_{\optTar{X}}(0,1)} &\leq
\sup_{\substack{\|f\|_{\Lambda_X(0,1)}\leq1 \\ f\in\F}} \Big\| \int_t^1 f(s) s^{-1 + \frac{m}{n}} \dd{s} \Big\|_{M_{\optTar{X}}(0,1)} \nonumber\\
&\leq C \sup_{\substack{\|h\|_{\Lambda_X(0,1)}\leq1 \\ h\in\I_{a/2,a}}} \Big\| \int_t^1 h(s) s^{-1 + \frac{m}{n}} \dd{s} \Big\|_{M_{\optTar{X}}(0,1)} \label{E:key_inequality2}
\end{align}
for every $a\in(0,1]$, where $Y_X$ is the optimal target r.i.\ space for $X$ in the Sobolev embedding \eqref{prel:optimal_Sob_emb}.
\end{proposition}
\begin{proof}
The first inequality in \eqref{E:key_inequality2} can be readily verified. Indeed, if $h = \alpha \chi_{(a/2, a)} \in \I_{a/2,a}$ is such that $\|h\|_{\Lambda_X(0,1)}\leq1$, then
\begin{align*}
\Big\| \int_t^1 h(s) s^{-1 + \frac{m}{n}} \dd{s} \Big\|_{M_{\optTar{X}}(0,1)} &\leq \Big\| \int_t^1 \alpha \chi_{(0, a)}(s)  s^{-1 + \frac{m}{n}} \dd{s} \Big\|_{M_{\optTar{X}}(0,1)} \\
&\leq\sup_{\substack{\|f\|_{\Lambda_X(0,1)}\leq1 \\ f\in\F}} \Big\| \int_t^1 f(s) s^{-1 + \frac{m}{n}} \dd{s} \Big\|_{M_{\optTar{X}}(0,1)}.
\end{align*}
Therefore, we only need to prove the second inequality in \eqref{E:key_inequality2} with a constant $C$ independent of $a\in(0,1]$. Let $f = \sum_{j = 1}^M \alpha_j \chi_{(0, b_j)}\in\F$ be such that $\|f\|_{\Lambda_X(0,1)}\leq1$. Set $b_0 = 0$. Using \eqref{E:H_on_simple}, \eqref{E:m_and_M_coincide_for_optimal_tar:1}, and the monotonicity of $\fundOptX$, we have
\begin{align*}
\Big\| \int_t^1 f(s) s^{-1 + \frac{m}{n}} \dd{s} \Big\|_{M_{\optTar{X}}(0,1)} &\approx \Big\| \sum_{j = 1}^M \alpha_j b_j^{\frac{m}{n}}\chi_{(0, b_j)} \Big\|_{M_{\optTar{X}}(0,1)} \\
&\approx \Big\| \sum_{j = 1}^M \alpha_j b_j^{\frac{m}{n}}\chi_{(0, b_j)} \Big\|_{m_{\optTar{X}}(0,1)} \\
&= \sup_{t\in(0,1)} \fundOptX(t) \sum_{j = 1}^M \alpha_j b_j^{\frac{m}{n}}\chi_{(0, b_j)}(t) \\
&= \max_{k = 1, \dots, M} \sup_{t\in(b_{k-1},b_k)} \fundOptX(t) \sum_{j = 1}^M \alpha_j b_j^{\frac{m}{n}}\chi_{(0, b_j)}(t) \\
&= \max_{k = 1, \dots, M} \fundOptX(b_k) \sum_{j = k}^M \alpha_j b_j^{\frac{m}{n}}.
\end{align*}
Let $k\in\{1,\dots, M\}$ be an index where the maximum is attained. Hence
\begin{equation}\label{E:1}
\Big\| \int_t^1 f(s) s^{-1 + \frac{m}{n}} \dd{s} \Big\|_{M_{\optTar{X}}(0,1)} \approx \fundOptX(b_k) \sum_{j = k}^M \alpha_j b_j^{\frac{m}{n}}.
\end{equation}
Let $a\in(0,1]$. Set
\begin{align*}
g &= \Big( \frac{\fundOptX(b_k)}{\fundOptX(a)}a^{-\frac{m}{n}} \sum_{j = k}^M \alpha_j b_j^{\frac{m}{n}} \Big) \chi_{(0,a)} \\
\intertext{and}
h &= g\chi_{(a/2, a)}.
\end{align*}
Note that
\begin{equation}\label{E:2}
h\in\I_{a/2,a}.
\end{equation}
Now, using \eqref{E:H_on_simple}, \eqref{E:m_and_M_coincide_for_optimal_tar:1}, and \eqref{E:1}, we see that
\begin{align}
\Big\| \int_t^1 g(s) s^{-1 + \frac{m}{n}} \dd{s} \Big\|_{M_{\optTar{X}}(0,1)} &\approx \Big\| \Big( \frac{\fundOptX(b_k)}{\fundOptX(a)}a^{-\frac{m}{n}} \sum_{j = k}^M \alpha_j b_j^{\frac{m}{n}} \Big) a^{\frac{m}{n}} \chi_{(0,a)} \Big\|_{M_{\optTar{X}}(0,1)} \nonumber\\
&\approx \Big\| \Big( \frac{\fundOptX(b_k)}{\fundOptX(a)} \sum_{j = k}^M \alpha_j b_j^{\frac{m}{n}} \Big) \chi_{(0,a)} \Big\|_{m_{\optTar{X}}(0,1)} \nonumber\\
&= \fundOptX(b_k) \sum_{j = k}^M \alpha_j b_j^{\frac{m}{n}} \nonumber\\
&\approx \Big\| \int_t^1 f(s) s^{-1 + \frac{m}{n}} \dd{s} \Big\|_{M_{\optTar{X}}(0,1)}, \label{E:3}
\end{align}
in which the multiplicative constants are independent of $a$. Furthermore, using the quasiconcavity of $\fundOptX$, we have
\begin{align*}
\Big\| \int_t^1 h(s) s^{-1 + \frac{m}{n}} \dd{s} \Big\|_{M_{\optTar{X}}(0,1)} &\geq \Big\| \chi_{(0, a/2)}(t) \int_t^1 h(s) s^{-1 + \frac{m}{n}} \dd{s} \Big\|_{M_{\optTar{X}}(0,1)} \\
&= \Big( \frac{\fundOptX(b_k)}{\fundOptX(a)}a^{-\frac{m}{n}} \sum_{j = k}^M \alpha_j b_j^{\frac{m}{n}} \Big) \Big\| \chi_{(0, a/2)}(t) \int_{a/2}^a s^{-1 + \frac{m}{n}} \dd{s} \Big\|_{M_{\optTar{X}}(0,1)} \\
&\approx \Big( \frac{\fundOptX(b_k)}{\fundOptX(a)}a^{-\frac{m}{n}} \sum_{j = k}^M \alpha_j b_j^{\frac{m}{n}} \Big) a^\frac{m}{n}\fundOptX\Big( \frac{a}{2} \Big) \\
&\geq \frac1{2} \Big( \frac{\fundOptX(b_k)}{\fundOptX(a)} \sum_{j = k}^M \alpha_j b_j^{\frac{m}{n}} \Big) \fundOptX(a) \\
&= \frac1{2} \Big\|\Big( \frac{\fundOptX(b_k)}{\fundOptX(a)}a^{-\frac{m}{n}} \sum_{j = k}^M \alpha_j b_j^{\frac{m}{n}} \Big) a^{\frac{m}{n}} \chi_{(0,a)} \Big\|_{M_{\optTar{X}}(0,1)}.
\end{align*}
Combining this and \eqref{E:3}, we obtain
\begin{equation}\label{E:6}
\Big\| \int_t^1 g(s) s^{-1 + \frac{m}{n}} \dd{s} \Big\|_{M_{\optTar{X}}(0,1)} \lesssim \Big\| \int_t^1 h(s) s^{-1 + \frac{m}{n}} \dd{s} \Big\|_{M_{\optTar{X}}(0,1)}.
\end{equation}

Next, using \eqref{E:sublimiting} twice and the fact that $b_j\geq b_k$ for every $j\in\{k,\dots, M\}$, we arrive at
\begin{align}
\|h\|_{\Lambda_X(0,1)} \leq \|g\|_{\Lambda_X(0,1)} &= \Big( \frac{\fundOptX(b_k)}{\fundOptX(a)}a^{-\frac{m}{n}} \sum_{j = k}^M \alpha_j b_j^{\frac{m}{n}} \Big) \fundX(a) \nonumber\\
&\approx \fundOptX(b_k) \sum_{j = k}^M \alpha_j b_j^{\frac{m}{n}} \nonumber\\
&\approx \fundOptX(b_k) \sum_{j = k}^M \alpha_j \frac{\fundX(b_j)}{\fundOptX(b_j)} \nonumber\\
&\leq \sum_{j = k}^M \alpha_j \fundX(b_j) = \Big\| \sum_{j = k}^M \alpha_j \chi_{(0, b_j)} \Big\|_{\Lambda_X(0,1)} \nonumber\\
&\leq \|f\|_{\Lambda_X(0,1)} \leq 1. \label{E:4}
\end{align}

Finally, combining \eqref{E:2}, \eqref{E:3}, \eqref{E:6}, and \eqref{E:4}, we obtain \eqref{E:key_inequality2} with a constant $C$ independent of $a$.
\end{proof}

We are finally in a position to prove the main result of this section.
\begin{theorem}\label{thm:small_support_big_norm}
Let $1\leq m < n$. Let $\|\cdot\|_{X(0,1)}$ be an r.i.\ function norm that is sublimiting. Let $\|\cdot\|_{Y(0,1)}$ be an r.i.\ function norm such that
\begin{equation}\label{E:target_between_optimal_and_Marc}
\optTar{X}(0,1) \hookrightarrow Y(0,1) \hookrightarrow M_{\optTar{X}}(0,1).
\end{equation}
Let $B(x_0, R) \subseteq \Omega$ be a ball inside $\Omega$. There are constants $C_1$ and $C_2$ such that for every $a\in(0,1)$ there are $f\in\I_{a/2,a}$ and a function $u_f\in\VmXOm\cap \mathcal C^{m-1}(\Omega\setminus\{x_0\})$ such that
\begin{align}
\|u_f\|_{\VmXOm} &\leq C_1 \|f\|_{X(0,1)} = C_1 \label{thm:small_support_big_norm:E:controlled_Sobolev_norm},\\
0 < C_2 &\leq \|u_f\|_{Y(\Omega)} < \infty \label{thm:small_support_big_norm:E:big_Y_norm},
\end{align}
and
\begin{equation}\label{thm:small_support_big_norm:E:support_of_uf}
\text{the support of $u_f$ is the closed ball $\bar{B}(x_0, a^\frac1{n}R)$}.
\end{equation}
Furthermore, if $\sum_{k = 1}^M \alpha_j u_{f_j}$ is a linear combination of such functions, we have
\begin{equation}\label{thm:small_support_big_norm:E:controlled_Sobolev_norm_of_linear_comb}
\Big\| \sum_{k = 1}^M \alpha_j u_{f_j} \Big\|_{\VmXOm} \leq C_1 \Big\| \sum_{k = 1}^M \alpha_j f_j \Big\|_{X(0,1)}.
\end{equation}
\end{theorem}
\begin{proof}
Without loss of generality, we may assume that $\fundX$ is concave. Set
\begin{equation*}
\widetilde{C_2} = \sup_{\|f\|_{\Lambda_X(0,1)}\leq1} \Big\| \int_t^1 f^*(s) s^{-1 + \frac{m}{n}} \dd{s} \Big\|_{M_{\optTar{X}}(0,1)}.
\end{equation*}
Since every nonnegative nonincreasing function on $(0,1)$ is an a.e.~pointwise limit of a nondecreasing sequence of nonnegative nonincreasing simple functions on $(0,1)$, it is easy to see that
\begin{equation*}
\widetilde{C_2} = \sup_{\substack{\|f\|_{\Lambda_X(0,1)}\leq1\\ f\in\F}} \Big\| \int_t^1 f(s) s^{-1 + \frac{m}{n}} \dd{s} \Big\|_{M_{\optTar{X}}(0,1)}.
\end{equation*}
Furthermore, we have
\begin{equation}\label{thm:small_support_big_norm:E:8}
\widetilde{C_2} \approx \sup_{\substack{\|f\|_{\Lambda_X(0,1)}\leq1\\ f\in\I_{a/2,a}}} \Big\| \int_t^1 f(s) s^{-1 + \frac{m}{n}} \dd{s} \Big\|_{M_{\optTar{X}}(0,1)} \quad \text{for every $a\in(0,1]$}
\end{equation}
thanks to \cref{prop:cutting_of_support}. By \cite[Theorem~A]{KP:06} combined with \eqref{E:target_between_optimal_and_Marc}, we have
\begin{equation}\label{thm:small_support_big_norm:E:9}
\sup_{\|f\|_{X(0,1)}\leq1} \Big\| \int_t^1 f(s) s^{-1 + \frac{m}{n}} \dd{s} \Big\|_{Y(0,1)} < \infty.
\end{equation}
Hence, combining \eqref{thm:small_support_big_norm:E:8}, \eqref{thm:small_support_big_norm:E:9}, \eqref{E:target_between_optimal_and_Marc}, \eqref{prel:fundamental_endpoint_embedding_LorLam_to_X}, and \eqref{prel:fundamental_endpoint_embedding_into_M}, we obtain
\begin{equation}\label{thm:small_support_big_norm:E:3}
\widetilde{C_2} \lesssim \sup_{\substack{\|f\|_{X(0,1)}\leq1\\ f\in\I_{a/2,a}}} \Big\| \int_t^1 f(s) s^{-1 + \frac{m}{n}} \dd{s} \Big\|_{Y(0,1)} < \infty \quad \text{for every $a\in(0,1]$}.
\end{equation}

Next, let $B(x_0, R) \subseteq \Omega$ be a ball inside $\Omega$. By following and suitably modifying the construction from \cite[p.~562]{KP:06} (cf.~\cite[Proposition~4.5]{M:25preprint}), we can construct for every $f\in L^\infty(0,1)$ that has support in $[0,\delta]$ for some $\delta\in(0,1)$ functions $u_f$ and $u_{|f|}$ such that both $u_f$ and $u_{|f|}$ belong to $\VmXOm\cap \mathcal C^{m-1}(\Omega\setminus\{x_0\})$, $u_{|f|}$ is radially nonincreasing with respect to $x_0$, $|u_f|\leq u_{|f|}$,
\begin{align}
\|u_f\|_{\VmXOm} &\leq C_1 \|f\|_{X(0,1)}, \label{thm:small_support_big_norm:E:1}\\
\Big\| \int_t^1 |f(s)| s^{-1+\frac{m}{n}} \dd{s} \Big\|_{Y(0,1)} &\leq C \|u_{|f|}\|_{Y(\Omega)} < \infty, \label{thm:small_support_big_norm:E:2}
\end{align}
and
\begin{equation}\label{thm:small_support_big_norm:E:6}
\text{the support of $u_{|f|}$ is the closed ball $\bar{B}(x_0, \varrho^\frac1{n}R)$},
\end{equation}
where
\begin{equation*}
\varrho = \esssup\{t\in(0,1): |f(t)| > 0 \},
\end{equation*}
and in which the constants $C_1$ and $C$ depend only on $m$, $n$, and $R$ (in particular, they are independent of $\delta$ and $f$). The fact that $u_{|f|}\in Y(\Omega)$ follows from \eqref{E:target_between_optimal_and_Marc}. Moreover, the assignment $f\mapsto u_f$ is linear, i.e., if $\sum_{k = 1}^M \alpha_j f_j$ is a linear combination of functions from $L^\infty(0,1)$ having supports in $[0, \delta]$ for some $\delta\in(0,1)$, then
\begin{equation}\label{thm:small_support_big_norm:E:7}
\sum_{k = 1}^M \alpha_j u_{f_j} = u_{\sum_{k = 1}^M \alpha_j f_j}.
\end{equation}

Now, let $a\in(0,1)$. Using \eqref{thm:small_support_big_norm:E:3}, we obtain $f\in\I_{a/2,a}$ such that
\begin{equation}\label{thm:small_support_big_norm:E:4}
\|f\|_{X(0,1)} = 1
\end{equation}
and
\begin{equation}\label{thm:small_support_big_norm:E:5}
\widetilde{C_2} \lesssim \Big\| \int_t^1 f(s) s^{-1 + \frac{m}{n}} \dd{s} \Big\|_{Y(0,1)}.
\end{equation}
Let $u_f\in \VmXOm$ be the function from \eqref{thm:small_support_big_norm:E:1}--\eqref{thm:small_support_big_norm:E:6}. Note that $f = |f|$, and so $u_f = u_{|f|}$. Since $f\in\I_{a/2,a}$, it follows from \eqref{thm:small_support_big_norm:E:6} that \eqref{thm:small_support_big_norm:E:support_of_uf} is true. By combing \eqref{thm:small_support_big_norm:E:4} and \eqref{thm:small_support_big_norm:E:1}, we obtain \eqref{thm:small_support_big_norm:E:controlled_Sobolev_norm}. Furthermore, using \eqref{thm:small_support_big_norm:E:5} and \eqref{thm:small_support_big_norm:E:2}, we see that \eqref{thm:small_support_big_norm:E:big_Y_norm} is also valid with a constant $C_2$ independent of $a$.

At last, if $\sum_{k = 1}^M \alpha_j u_{f_j}$ is a linear combination, where $f_j\in\I_{a_j/2,a_j}$ for some $a_j\in(0,1)$, then
\begin{equation*}
\Big\| \sum_{k = 1}^M \alpha_j u_{f_j} \Big\|_{\VmXOm} = \Big\| u_{\sum_{k = 1}^M \alpha_j f_j} \Big\|_{\VmXOm} \leq C_1 \Big\| \sum_{k = 1}^M \alpha_j f_j \Big\|_{X(0,1)}
\end{equation*}
thanks to \eqref{thm:small_support_big_norm:E:7} and \eqref{thm:small_support_big_norm:E:1}, whence \eqref{thm:small_support_big_norm:E:controlled_Sobolev_norm_of_linear_comb} follows.
\end{proof}

\begin{remark}
When $m=1$, the function $u_f$ associated to $f \in\I_{a/2,a}$ from the preceding theorem can be constructed in such a way that $u_f$ is constant in the ball $B(x_0, (a/2)^{1/n}R)$ and
\begin{equation*}
\nabla u_f(x) = \frac{n}{R} f\Big( \frac{|x-x_0|^n}{R^n} \Big)  \frac{x-x_0}{|x-x_0|}\chi_{B(x_0,R)}(x) \quad \text{for a.e.~$x\in\Omega$}.
\end{equation*}
In particular, if $\{(a_j/2, a_j)\}_{j=1}^M\subseteq(0,1)$ are mutually disjoint intervals, then the supports of $\{\nabla u_{f_j}\}_{j=1}^M$, where $f_j\in\I_{a_j/2,a_j}$, are also mutually disjoint.
\end{remark}

\section{Optimal embeddings of ``nonlimiting'' \texorpdfstring{$\Lambda$}{Lambda} spaces and their non-strict singularity} \label{Section4}
We start with an auxiliary  proposition of independent interest, which will be an essential ingredient in the proof of our main result, \cref{thm:optimal_embedding_Lambda_spaces_not_SS}. Loosely speaking, it tells us that when $X$ is a Lambda space, the optimal r.i.\ space in Sobolev embeddings \eqref{prel:optimal_Sob_emb} that are not (super)limiting is a Lambda space of the same type. The price for the generality of \cref{prop:optimal_for_Lambda_is_Lambda} is that its assumptions are unavoidably technical. Nevertheless, it is usually not difficult to verify them (see~\cref{rem:optimal_for_Lambda_is_Lambda}).
\begin{proposition}\label{prop:optimal_for_Lambda_is_Lambda}
Let $1\leq m <n$. Let $q\in[1, \infty)$ and let $w\in\Mpl(0,1)$ be a weight satisfying \eqref{prel:when_Lambda_equiv_ri}. Furthermore, assume that
\begin{align}
\sup_{t\in(0,1)} \frac{t^{\frac{m}{n}}}{W(t)} &= \infty \quad \text{if $q = 1$}, \label{prop:optimal_for_Lambda_is_Lambda:eq:not_limiting_q=1} \\
\int_0^1 \frac{w(t)}{W(t)^{q'}}t^{\frac{m}{n}q'} \dd{t} &= \infty \quad \text{if $q\in (1 ,\infty)$}. \label{prop:optimal_for_Lambda_is_Lambda:eq:not_limiting_q_bigger_1}
\end{align}
When $q\in(1,\infty)$, assume in addition that
\begin{align}
\int_0^{t/2} \frac{w(s)}{W(s)^{q'}} s^{q'} \dd{s} &\lesssim \int_{t/2}^{t} \frac{w(s)}{W(s)^{q'}} s^{q'} \dd{s} \label{prop:optimal_for_Lambda_is_Lambda:eq:q_bigger_1_not_close_L1}\\
\intertext{and}
\int_0^t \Big(\int_0^s \frac{w(\tau)}{W(\tau)^{q'}} \tau^{q'} \dd{\tau} \Big)^{1-q} s^{q-1} \dd {s} &\lesssim t^q \Big(\int_0^t \frac{w(s)}{W(s)^{q'}} s^{q'} \dd{s} \Big)^{1-q} \label{prop:optimal_for_Lambda_is_Lambda:eq:q_bigger_1_not_close_Linfty}
\end{align}
for every $t\in(0,1)$. Then there is a weight $w_{opt}$ such that the optimal target r.i.\ space $Y_X(\Omega)$ for $X(\Omega) = \Lambda^q_w(\Omega)$ in the Sobolev embedding \eqref{prel:optimal_Sob_emb} satisfies
\begin{equation}\label{prop:optimal_for_Lambda_is_Lambda:eq:optimal_is_Lambda}
Y_X(\Omega) = \Lambda^q_{w_{opt}}(\Omega).
\end{equation}
\end{proposition}
\begin{proof}
We start by observing that the assumptions \eqref{prop:optimal_for_Lambda_is_Lambda:eq:not_limiting_q=1} and \eqref{prop:optimal_for_Lambda_is_Lambda:eq:not_limiting_q_bigger_1} are equivalent to the fact that
\begin{equation}\label{prop:optimal_for_Lambda_is_Lambda:eq:6}
\lim_{t\to0^+}\fundOptX(t) = 0.
\end{equation}
This follows from the following. By \cite[Proposition~1]{S:93} (see also~\cite[p.~148]{S:90}), the assumptions \eqref{prop:optimal_for_Lambda_is_Lambda:eq:not_limiting_q=1} and \eqref{prop:optimal_for_Lambda_is_Lambda:eq:not_limiting_q_bigger_1} are equivalent to the fact that
\begin{equation*}
\Lambda^q_w(\Omega) \not\subseteq L^{\frac{n}{m},1} (\Omega).
\end{equation*}
As was explained in the proof of Proposition~\ref{prop:sublimiting_X_neces_conds}, this is in turn equivalent to the validity of \eqref{prop:optimal_for_Lambda_is_Lambda:eq:6}.

First, let $q=1$. Since the weight $w$ satisfies \eqref{prel:when_Lambda_equiv_ri}, we may assume without loss of generality that the function $W$ is concave. Hence, it follows from \eqref{prel:asoc_of_endpoint_Lorentz_is_M} that
\begin{equation}\label{prop:optimal_for_Lambda_is_Lambda:eq:12}
\big( \Lambda^1_w \big)'(0,1) = \big( \Lambda_W \big)'(0,1) = M_{\frac{t}{W(t)}}(0,1).
\end{equation}
Furthermore, we claim that
\begin{equation}\label{prop:optimal_for_Lambda_is_Lambda:eq:11}
M_{\frac{t}{W(t)}}(0,1) = m_{\frac{t}{W(t)}}(0,1).
\end{equation}
To this end, using \eqref{E:sublimiting} twice, we have
\begin{align*}
\frac1{t} \int_0^t \frac{W(s)}{s} \dd{s} &\approx \frac1{t} \int_0^t \fundOptX(s) s^{-1 + \frac{m}{n}} \dd{s} \leq \frac{\fundOptX(t)}{t} \int_0^t s^{-1 + \frac{m}{n}} \dd{s} \\
&\approx \frac{\fundOptX(t)}{t} t^\frac{m}{n} \approx \frac{W(t)}{t}
\end{align*}
for every $t\in(0,1)$. In other words, \eqref{prel:M=m_int_cond} with $\varphi(t) = t/W(t)$, $t\in(0, 1)$, is true. It follows that \eqref{prop:optimal_for_Lambda_is_Lambda:eq:11} holds. By combining \eqref{prop:optimal_for_Lambda_is_Lambda:eq:12} and \eqref{prop:optimal_for_Lambda_is_Lambda:eq:11}, we arrive at
\begin{equation}\label{prop:optimal_for_Lambda_is_Lambda:eq:1}
\big( \Lambda^1_w \big)'(0,1) = m_{\frac{t}{W(t)}}(0,1).
\end{equation}
Next, we claim that
\begin{equation}\label{prop:optimal_for_Lambda_is_Lambda:eq:2}
\|g\|_{Y_X'(0,1)} \approx \sup_{t\in(0,1)}t^{\frac{m}{n}}\frac{t}{W(t)} g^{**}(t) \quad \text{for every $g\in\M(0,1)$}.
\end{equation}
To this end, using \eqref{prel:optimal_asoc_norm}, \eqref{prop:optimal_for_Lambda_is_Lambda:eq:1}, and the monotonicity of the function $t/W(t)$, we have
\begin{align}
\|g\|_{Y_X'(0,1)} &\approx \sup_{t\in(0,1)} \frac{t}{W(t)} \big[ s^{\frac{m}{n}} g^{**}(s) \big]^*(t) \leq \sup_{t\in(0,1)} \frac{t}{W(t)} \sup_{t\leq s < 1} s^{\frac{m}{n}} g^{**}(s) \nonumber\\
&=\sup_{s\in(0,1)} s^{\frac{m}{n}}g^{**}(s) \sup_{0< t \leq s} \frac{t}{W(t)} = \sup_{s\in(0,1)}s^{\frac{m}{n}}\frac{s}{W(s)} g^{**}(s) \label{prop:optimal_for_Lambda_is_Lambda:eq:3}
\end{align}
for every $g\in\M(0,1)$. As for the reverse inequality, we have
\begin{align}
\sup_{t\in(0,1)} \frac{t}{W(t)} \sup_{t\leq s < 1} s^{\frac{m}{n}} g^{**}(s) &= \|\sup_{t\leq s < 1} s^{\frac{m}{n}} g^{**}(s)\|_{m_{\frac{t}{W(t)}}(0,1)} \nonumber\\
&\lesssim \|t^{\frac{m}{n}} g^{**}(t)\|_{m_{\frac{t}{W(t)}}(0,1)} \approx \|g\|_{Y_X'(0,1)} \label{prop:optimal_for_Lambda_is_Lambda:eq:4}
\end{align}
for every $g\in\M(0,1)$ thanks to \cite[Lemma~4.10]{EMMP:20}, \eqref{prop:optimal_for_Lambda_is_Lambda:eq:1}, and \eqref{prel:optimal_asoc_norm}. Therefore, combining \eqref{prop:optimal_for_Lambda_is_Lambda:eq:3} and \eqref{prop:optimal_for_Lambda_is_Lambda:eq:4}, we obtain \eqref{prop:optimal_for_Lambda_is_Lambda:eq:2}. Now, using \eqref{prop:optimal_for_Lambda_is_Lambda:eq:2} and \eqref{E:sublimiting}, we see that
\begin{equation}\label{prop:optimal_for_Lambda_is_Lambda:eq:5}
\|g\|_{Y_X'(0,1)} \approx \sup_{t\in(0,1)}\frac{t}{\fundOptX(t)}  g^{**}(t) = \|g\|_{M_{\frac{t}{\fundOptX(t)}}} \quad \text{for every $g\in\M(0,1)$}.
\end{equation}
Hence, combining \eqref{prop:optimal_for_Lambda_is_Lambda:eq:5} and \eqref{prel:asoc_of_M_is_endpoint_Lorentz}, we obtain
\begin{equation}\label{prop:optimal_for_Lambda_is_Lambda:eq:8}
Y_X(0,1) = \Lambda_{\tilde{\psi}}(0,1),
\end{equation}
where $\tilde{\psi}$ is the least nondecreasing concave majorant of $\fundOptX$. Finally, \eqref{prop:optimal_for_Lambda_is_Lambda:eq:8} together with \eqref{prop:optimal_for_Lambda_is_Lambda:eq:6} and \eqref{prel:concave_majorant_pointwise_equiv} implies that
\begin{equation*}
Y_X(0,1) =  \Lambda^1_{w_{opt}}(0,1)
\end{equation*}
for some weight $w_{opt}$\textemdash namely, $w_{opt}=\tilde{\psi}'$.

Now, assume that $q\in (1, \infty)$. By \cite[Theorem~1]{S:90}, we have
\begin{equation}\label{prop:optimal_for_Lambda_is_Lambda:eq:13}
\|g\|_{(\Lambda^q_w(0,1))'} \approx \|g^{**}\|_{\Lambda^{q'}_v(0,1)} + \int_0^1 g^*(s) \dd {s} \quad \text {for every $g\in\M(0,1)$}, 
\end{equation}
where
\begin{equation*}
v(t) = \frac{w(t)}{W(t)^{q'}} t^{q'},\ t\in(0,1).
\end{equation*}
Furthermore, note that \eqref{prop:optimal_for_Lambda_is_Lambda:eq:q_bigger_1_not_close_Linfty} is \eqref{prel:when_Lambda_equiv_ri} for $\Lambda^{q'}_v(0,1)$. It follows from \cite[p.~148 and Theorem~4]{S:90} that
\begin{equation}\label{prop:optimal_for_Lambda_is_Lambda:eq:9}
\|g^{**}\|_{\Lambda^{q'}_v(0,1)} + \int_0^1 g^*(s) \dd {s} \approx \|g\|_{\Lambda^{q'}_v(0,1)} \quad \text {for every $g\in\M(0,1)$}.
\end{equation}
Hence, by combining \eqref{prop:optimal_for_Lambda_is_Lambda:eq:13} and \eqref{prop:optimal_for_Lambda_is_Lambda:eq:9}, we obtain
\begin{equation*}
X'(0,1) = (\Lambda^q_w(0,1))' = \Lambda^{q'}_v(0,1).
\end{equation*}
Therefore, we have
\begin{equation}\label{prop:optimal_for_Lambda_is_Lambda:eq:10}
\|g\|_{Y_X'(0,1)} \approx \|t^{\frac{m}{n}} g^{**}(t)\|_{\Lambda^{q'}_v(0,1)} \quad \text{for every $g\in \M(0,1)$}
\end{equation}
thanks to \eqref{prel:optimal_asoc_norm}. Now, on the one hand, we have
\begin{equation}\label{prop:optimal_for_Lambda_is_Lambda:eq:14}
\|t^{\frac{m}{n}} g^{**}(t)\|_{\Lambda^{q'}_v(0,1)} \leq \|v(t)^{\frac1{q'}} \sup_{t\leq s\leq 1} s^{\frac{m}{n}} g^{**}(s)\|_{L^{q'}(0,1)}
\end{equation}
for every $g\in \M(0,1)$. Since
\begin{align*}
t^{\frac{m}{n}q'} \int_0^t v(s) \dd{s} &\lesssim t^{\frac{m}{n}q'} \int_t^{2t} v(s) \dd{s} \leq \int_t^{2t} s^{\frac{m}{n}q'} v(s) \dd{s} \\
&= \sum_{j=0}^\infty \int_{2^{-j}t}^{2^{-j+1}t} s^{\frac{m}{n}q'} v(s) \dd{s} \lesssim \sum_{j=0}^\infty (2^{-j-1}t)^{\frac{m}{n}q'} \int_{2^{-j}t}^{2^{-j+1}t}  v(s) \dd{s} \\
&\leq \sum_{j=0}^\infty (2^{-j-1}t)^{\frac{m}{n}q'} V(2^{-j+1}t) \lesssim \sum_{j=0}^\infty (2^{-j-1}t)^{\frac{m}{n}q'} V(2^{-j}t) \\
&\lesssim \sum_{j=0}^\infty (2^{-j-1}t)^{\frac{m}{n}q'} V(2^{-j-1}t) \lesssim \sum_{j=0}^\infty (2^{-j-1}t)^{\frac{m}{n}q'} \int_{2^{-j-1}t}^{2^{-j}t} v(s) \dd{s} \\
&\leq \sum_{j=0}^\infty \int_{2^{-j-1}t}^{2^{-j}t} s^{\frac{m}{n}q'}v(s) \dd{s} = \int_0^t s^{\frac{m}{n}q'}v(s) \dd{s}
\end{align*}
for every $t\in(0, 1/2)$ thanks to \eqref{prop:optimal_for_Lambda_is_Lambda:eq:q_bigger_1_not_close_L1} and \eqref{prel:Lambda_delta2} (with $W$ replaced by $V$), it follows from \cite[Theorem~3.2]{GOP:06} that
\begin{equation*}
\|v(t)^{\frac1{q'}} \sup_{t\leq s\leq 1} s^{\frac{m}{n}} g^{**}(s)\|_{L^{q'}(0,1)} \lesssim \|t^{\frac{m}{n}}v(t)^{\frac1{q'}} g^{**}(t)\|_{L^{q'}(0,1)} \quad \text{for every $g\in \M(0,1)$}.
\end{equation*}
Combining this with \eqref{prop:optimal_for_Lambda_is_Lambda:eq:14}, we obtain
\begin{equation}\label{prop:optimal_for_Lambda_is_Lambda:eq:15}
\|t^{\frac{m}{n}} g^{**}(t)\|_{\Lambda^{q'}_v(0,1)} \lesssim \|t^{\frac{m}{n}}v(t)^{\frac1{q'}} g^{**}(t)\|_{L^{q'}(0,1)} \quad \text{for every $g\in \M(0,1)$}.
\end{equation}
On the other hand, as in \eqref{prop:optimal_for_Lambda_is_Lambda:eq:4}, we have
\begin{align}
\|t^{\frac{m}{n}}v(t)^{\frac1{q'}} g^{**}(t)\|_{L^{q'}(0,1)} &\leq \|v(t)^{\frac1{q'}} \sup_{t\leq s\leq1}s^\frac{m}{n}g^{**}(s)\|_{L^{q'}(0,1)} \nonumber\\
&= \|\sup_{t\leq s\leq1}s^\frac{m}{n}g^{**}(s)\|_{\Lambda^{q'}_v(0,1)} \nonumber\\
&\lesssim  \|t^\frac{m}{n}g^{**}(t)\|_{\Lambda^{q'}_v(0,1)} \label{prop:optimal_for_Lambda_is_Lambda:eq:16}
\end{align}
for every $g\in \M(0,1)$ thanks to \cite[Lemma~4.10]{EMMP:20}. Hence, by combining \eqref{prop:optimal_for_Lambda_is_Lambda:eq:10}, \eqref{prop:optimal_for_Lambda_is_Lambda:eq:15}, and \eqref{prop:optimal_for_Lambda_is_Lambda:eq:16}, we obtain
\begin{equation*}
\|g\|_{Y_X'(0,1)} \approx \|t^{\frac{m}{n}}v(t)^{\frac1{q'}} g^{**}(t)\|_{L^{q'}(0,1)} \quad \text{for every $g\in \M(0,1)$}.
\end{equation*}
With this at our disposal, it follows from \cite[Theorem~2.7]{EKP:00} (we also use both \eqref{prop:optimal_for_Lambda_is_Lambda:eq:not_limiting_q_bigger_1} and \eqref{prop:optimal_for_Lambda_is_Lambda:eq:q_bigger_1_not_close_L1} here) that
\begin{equation*}
Y_X(0,1) = \Lambda^q_{w_{opt}}(0,1)
\end{equation*}
for some weight $w_{opt}$\textemdash namely 
\begin{equation*}
w_{opt}(t) = \Bigg( \Big( 1 + \int_{s}^1 \frac{w(\tau)}{W(\tau)^{q'}} \tau^{\frac{m}{n}q'} \dd{\tau} \Big)^{1-q} \Bigg)'(t) \quad \text{for a.e.~$t\in(0,1)$},
\end{equation*}
which concludes the proof.
\end{proof}

\pagebreak
\begin{remark}\label{rem:optimal_for_Lambda_is_Lambda}\ 
\begin{enumerate}[label=(\roman*), ref=(\roman*)]
\item As explained at the beginning of the proof of \cref{prop:optimal_for_Lambda_is_Lambda}, the assumptions \eqref{prop:optimal_for_Lambda_is_Lambda:eq:not_limiting_q=1} and \eqref{prop:optimal_for_Lambda_is_Lambda:eq:not_limiting_q_bigger_1} are equivalent to the fact that $X = \Lambda^q_w$ is not (super)limiting, by which we mean that $\VmXOm$ is not embedded in $L^\infty(\Omega)$. When $q\in(1,\infty)$, the assumptions \eqref{prop:optimal_for_Lambda_is_Lambda:eq:q_bigger_1_not_close_L1} and \eqref{prop:optimal_for_Lambda_is_Lambda:eq:q_bigger_1_not_close_Linfty} ensure, loosely speaking, that $X = \Lambda^q_w$ is not ``too close to $L^1$'' and not ``too close to $L^\infty$'', respectively. Typically, the validity of \eqref{prop:optimal_for_Lambda_is_Lambda:eq:q_bigger_1_not_close_Linfty} is already implied by that of \eqref{prop:optimal_for_Lambda_is_Lambda:eq:not_limiting_q_bigger_1}.
\item Consider the Lorentz spaces $L^{p,q}(0,1) = \Lambda^q_{w_{p,q}}(0,1)$ (recall~\eqref{prel:Lorentz_Lpq_weight}). Note that in order that $w$ is a weight, we need to have $p<\infty$. Moreover, when $q\in(1, \infty)$, \eqref{prop:optimal_for_Lambda_is_Lambda:eq:q_bigger_1_not_close_Linfty} is satisfied for $p<\infty$. It is easy to see that  \eqref{prel:when_Lambda_equiv_ri} is satisfied if and only if either $p\in(1, \infty)$ and $q\in[1, \infty)$ or $p=q=1$. When $q=1$, the condition \eqref{prop:optimal_for_Lambda_is_Lambda:eq:not_limiting_q=1} is satisfied if and only if $p < n/m$. When $q\in(1, \infty)$, the condition \eqref{prop:optimal_for_Lambda_is_Lambda:eq:not_limiting_q_bigger_1} is satisfied if and only if $p \leq n/m$, whereas \eqref{prop:optimal_for_Lambda_is_Lambda:eq:q_bigger_1_not_close_L1} is satisfied if and only if $p>1$. Putting everything together, we see that the assumptions of \cref{prop:optimal_for_Lambda_is_Lambda} are satisfied if and only if either $q=1$ and $p\in[1, n/m)$ or $q\in(1, \infty)$ and $p\in(1, n/m]$.
\item The assumptions of \cref{prop:optimal_for_Lambda_is_Lambda} do not guarantee that $\|\cdot\|_{X(0,1)} = \|\cdot\|_{\Lambda^q_w(0,1)}$ is sublimiting in the sense of \cref{def:sublimiting}. For example, they are satisfied for the Lorentz spaces $L^{\frac{n}{m},q}(0,1) = \Lambda^q_{w_{n/m,q}}(0,1)$ with $q\in(1, \infty)$ (see the preceding remark and \cref{rem:borderling_limiting_cases}).
\end{enumerate}
\end{remark}

The following proposition is essentially a particular case of \cref{prop:optimal_for_Lambda_is_Lambda} with the extra information that $\|\cdot\|_{X(0,1)} = \|\cdot\|_{\Lambda^q_w(0,1)}$ is sublimiting in the sense of \cref{def:sublimiting}.
\begin{proposition}\label{prop:when_is_optimal_Lambda_nonlimiting}
Let $1\leq m <n$. Let $q\in[1, \infty)$ and let $w\in\Mpl(0,1)$ be a weight satisfying \eqref{prel:when_Lambda_equiv_ri} and
\begin{equation}\label{prop:optimal_nonlimiting_Lambda:eq:not_limiting_int_cond}
\int_t^1 \frac{s^{-1 + \frac{m}{n}}}{W(s)^\frac1{q}} \dd{s} \lesssim \frac{t^\frac{m}{n}}{W(t)^\frac1{q}} \quad \text{for every $t\in(0,1)$}.
\end{equation}
When $q\in(1,\infty)$, assume in addition that \eqref{prop:optimal_for_Lambda_is_Lambda:eq:q_bigger_1_not_close_L1} and \eqref{prop:optimal_for_Lambda_is_Lambda:eq:q_bigger_1_not_close_Linfty} are also satisfied. Then $\|\cdot\|_{X(0,1)} = \|\cdot\|_{\Lambda^q_w(0,1)}$ is sublimiting and \eqref{prop:optimal_for_Lambda_is_Lambda:eq:optimal_is_Lambda} is satisfied for a weight $w_{opt}$ satisfying
\begin{equation}\label{prop:when_is_optimal_Lambda_nonlimiting:equivalence_of_fund_functions}
W_{opt}(t) \approx  t^{-\frac{m}{n}q}W(t) \quad \text{for every $t\in(0,1)$}.
\end{equation}
\end{proposition}
\begin{proof}
In view of \cref{prop:optimal_for_Lambda_is_Lambda}, \cref{rem:sublimiting_norms}\ref{rem:sublimiting_norms:item:nonlimiting_int_cond}, and \eqref{E:sublimiting}, we only need to show that the validity of \eqref{prop:optimal_nonlimiting_Lambda:eq:not_limiting_int_cond} implies that of \eqref{prop:optimal_for_Lambda_is_Lambda:eq:not_limiting_q=1} or \eqref{prop:optimal_for_Lambda_is_Lambda:eq:not_limiting_q_bigger_1}. Suppose that this is not the case. As explained at the beginning of the proof of \cref{prop:optimal_for_Lambda_is_Lambda}, this means that the embedding
\begin{equation}\label{prop:when_is_optimal_Lambda_nonlimiting:eq:1}
\Lambda^q_w(0,1) \hookrightarrow L^{\frac{n}{m},1} (0,1) \qquad \text{is true}.
\end{equation}
Furthermore, as in the proof of \cref{prop:sublimiting_X_neces_conds}, the validity of \eqref{prop:when_is_optimal_Lambda_nonlimiting:eq:1} is equivalent to the fact that $Y_X = L^\infty$. It follows that
\begin{equation*}
\fundOptX(t) \approx 1 \quad \text{for every $t\in(0,1)$}.
\end{equation*}
Combining this, the fact that $\|\cdot\|_{X(0,1)} = \|\cdot\|_{\Lambda^q_w(0,1)}$ is sublimiting (recall \cref{rem:sublimiting_norms}\ref{rem:sublimiting_norms:item:nonlimiting_int_cond}), and \eqref{E:sublimiting}, we obtain
\begin{equation}\label{prop:when_is_optimal_Lambda_nonlimiting:eq:2}
\frac{t^\frac{m}{n}}{W(t)^\frac1{q}} \approx 1 \quad \text{for every $t\in(0,1)$}.
\end{equation}
Now, using \eqref{prop:when_is_optimal_Lambda_nonlimiting:eq:2}, we have
\begin{equation}\label{prop:when_is_optimal_Lambda_nonlimiting:eq:3}
\int_t^1 \frac{s^{-1 + \frac{m}{n}}}{W(s)^\frac1{q}} \dd{s} \approx \int_t^1 s^{-1} \dd{s} = \log(1/t) \quad \text{for every $t\in (0,1)$}.
\end{equation}
However, \eqref{prop:when_is_optimal_Lambda_nonlimiting:eq:2} and \eqref{prop:when_is_optimal_Lambda_nonlimiting:eq:3} together clearly contradict the validity of \eqref{prop:optimal_nonlimiting_Lambda:eq:not_limiting_int_cond}, which concludes the proof.
\end{proof}

\begin{remark}
A simple but often applicable sufficient condition for \eqref{prop:optimal_nonlimiting_Lambda:eq:not_limiting_int_cond} is that there is $\alpha > m/n$ such that the function
		\begin{equation*}
			(0,1)\ni t\mapsto \frac{t^\alpha}{W(t)^\frac1{q}}
		\end{equation*}
		is equivalent to a nonincreasing function.
\end{remark}

We are finally ready to prove our main result, which loosely speaking tells us that the optimal sublimiting Sobolev embeddings \eqref{prel:optimal_Sob_emb} for Lambda spaces are not strictly singular. Its proof rests on \cref{thm:small_support_big_norm} and \cref{prop:when_is_optimal_Lambda_nonlimiting}, which are used in such a way that we can exploit the fact that both $X = \Lambda^q_{w}$ and $Y_X = \Lambda^q_{w_{opt}}$ contain an order isomorphic copy of $\ell_q$ (see~\cite[Theorem~1]{KM:04}).
\begin{theorem}\label{thm:optimal_embedding_Lambda_spaces_not_SS}
Let $1\leq m <n$. Let $q\in[1, \infty)$ and let $w\in\Mpl(0,1)$ be a weight satisfying \eqref{prel:when_Lambda_equiv_ri} and \eqref{prop:optimal_nonlimiting_Lambda:eq:not_limiting_int_cond}. When $q\in(1,\infty)$, assume in addition that \eqref{prop:optimal_for_Lambda_is_Lambda:eq:q_bigger_1_not_close_L1} and \eqref{prop:optimal_for_Lambda_is_Lambda:eq:q_bigger_1_not_close_Linfty} are also satisfied. Then the optimal target r.i.\ space $Y_X(\Omega)$ for $X(\Omega) = \Lambda^q_w(\Omega)$ in the Sobolev embedding \eqref{prel:optimal_Sob_emb} satisfies \eqref{prop:optimal_for_Lambda_is_Lambda:eq:optimal_is_Lambda} with a weight $w_{opt}$ satisfying \eqref{prop:when_is_optimal_Lambda_nonlimiting:equivalence_of_fund_functions}, and the optimal Sobolev embedding
\begin{equation}\label{thm:optimal_embedding_Lambda_spaces_not_SS:optimal_embedding}
\VmLamOm \hookrightarrow \Lambda^q_{w_{opt}}(\Omega)
\end{equation}
is not strictly singular. 

Moreover, there are constants $C_2$ and $C_3$ such that for all $\varepsilon_1, \varepsilon_2>0$ there is a linear independent sequence of functions $\{u_j\}_{j = 1}^\infty\subseteq \VmLamOm$ such that
\begin{equation}\label{thm:optimal_embedding_Lambda_spaces_not_SS:optimal_norm_E}
\Big\| \sum_{j=1}^\infty \alpha_j u_j \Big\|_{\Lambda^q_{w_{opt}}(\Omega)} \geq \Big( \frac{C_2}{1+\varepsilon_1} - \varepsilon_2 \Big) \Big( \sum_{j = 1}^\infty |\alpha_j|^q \Big)^\frac1{q}
\end{equation}
for every $\{\alpha_j\}_{j = 1}^\infty\in\ell_q$, and simultaneously
\begin{equation}\label{thm:optimal_embedding_Lambda_spaces_not_SS:gradient_norm_E}
\Big\| \sum_{j=1}^N \alpha_j u_j \Big\|_{\VmLamOm} \leq C_3 \Big( \sum_{j = 1}^N |\alpha_j|^q \Big)^\frac1{q}
\end{equation}
for every $N\in\N$.
\end{theorem}
\begin{proof}
By \cref{prop:when_is_optimal_Lambda_nonlimiting}, $\|\cdot\|_{X(0,1)} = \|\cdot\|_{\Lambda^q_w(0,1)}$ is sublimiting and the optimal target r.i.\ space $Y_X(\Omega)$ satisfies \eqref{prop:optimal_for_Lambda_is_Lambda:eq:optimal_is_Lambda} with a weight $w_{opt}$ for which \eqref{prop:when_is_optimal_Lambda_nonlimiting:equivalence_of_fund_functions} is true. Furthermore, it is easy to see that the fact that the optimal embedding \eqref{thm:optimal_embedding_Lambda_spaces_not_SS:optimal_embedding} is not strictly singular follows from the existence of a linear independent sequence of functions $\{u_j\}_{j = 1}^\infty\subseteq \VmLamOm$ for which \eqref{thm:optimal_embedding_Lambda_spaces_not_SS:optimal_norm_E} and \eqref{thm:optimal_embedding_Lambda_spaces_not_SS:gradient_norm_E} are true (with an arbitrarily chosen $\varepsilon_1, \varepsilon_2>0$ in such a way that $C_2/(1+\varepsilon_1) - \varepsilon_2 > 0$). Therefore, we only need to prove the existence of such a sequence of functions.

Let $C_1$ and $C_2$ be the constants from \eqref{thm:small_support_big_norm:E:controlled_Sobolev_norm} and \eqref{thm:small_support_big_norm:E:big_Y_norm} for $X = \Lambda^q_w$ and $Y = \Lambda^q_{w_{opt}}$. For every $j\in\N$, set
\begin{equation*}
\eta_j = \begin{cases}
	\varepsilon_2 \quad &\text{if $q=1$},\\
	\theta^j \quad &\text{if $q\in(1, \infty)$},
\end{cases}
\end{equation*}
where $\theta\in(0,1)$ is so small that
\begin{equation*}
\sum_{j=1}^\infty \theta^{q'j} \leq \varepsilon_2^{q'}.
\end{equation*}
Note that
\begin{equation}\label{thm:optimal_embedding_Lambda_spaces_not_SS:eq:23}
\|\{\eta_j\}_{j = 1}^\infty\|_{\ell_{q'}} \leq \varepsilon_2.
\end{equation}
Furthermore, set (recall~\eqref{prel:Lambda_delta2})
\begin{equation}\label{thm:optimal_embedding_Lambda_spaces_not_SS:eq:30}
\Delta_2 = \sup_{t\in(0,1/2)} \frac{W(2t)}{W(t)} < \infty.
\end{equation}

Fix a ball $B=B(x_0,R)\subseteq \Omega$ inside $\Omega$. Let $\varepsilon_1, \varepsilon_2 > 0$ be given. We will construct the desired linear independent sequence of functions $\{u_j\}_{j = 1}^\infty\subseteq \VmLamOm$ inductively. Set $a_1 = 1/2$ and $B_1 = B(x_0, a_1^{1/n}R)$. By~\cref{thm:small_support_big_norm}, there are positive constants $C_1$ and $C_2$, and functions $f_1\in\I_{a_1/2, a_1}$ and $u_1=u_{f_1}\in\VmXOm\cap\mathcal C^{m-1}(\Omega\setminus\{x_0\})$ such that
\begin{align}
f_1 &= c_1\chi_{(a_1/2, a_1)}, \label{thm:optimal_embedding_Lambda_spaces_not_SS:eq:1}\\
\|f_1\|_{\Lambda^q_w(0,1)} &= 1, \label{thm:optimal_embedding_Lambda_spaces_not_SS:eq:2}\\
\|u_1\|_{\VmLamOm} &\leq C_1 \|f_1\|_{\Lambda^q_w(0,1)} = C_1, \label{thm:optimal_embedding_Lambda_spaces_not_SS:eq:3}\\
C_2 &\leq \|u_1\|_{\Lambda^q_{w_{opt}}(\Omega)} < \infty, \label{thm:optimal_embedding_Lambda_spaces_not_SS:eq:4}
\end{align}
and
\begin{equation}\label{thm:optimal_embedding_Lambda_spaces_not_SS:eq:41}
    \text{the support of $u_1$ is the closed ball $\bar{B}_1$}.
\end{equation}
Note that
\begin{equation}\label{thm:optimal_embedding_Lambda_spaces_not_SS:eq:18}
\|f_1\|_{\Lambda^q_w(0,1)} = c_1 W\Big( \frac{a_1}{2} \Big)^\frac1{q} = 1
\end{equation}
thanks to \eqref{thm:optimal_embedding_Lambda_spaces_not_SS:eq:1} and \eqref{thm:optimal_embedding_Lambda_spaces_not_SS:eq:2}. Furthermore, thanks to the absolute continuity of $\|\cdot\|_{\Lambda^q_{w_{opt}}(\Omega)}$, the monotone convergence theorem, \eqref{thm:optimal_embedding_Lambda_spaces_not_SS:eq:18}, and the absolute continuity of the integral, we can find $0< \tilde{a} < a_1/4$ and $\tilde{B}_2 = B(x_0, \tilde{a}^{1/n}R)$ so small that
\begin{align}
\|u_1\chi_{\tilde{B}_2}\|_{\Lambda^q_{w_{opt}}(\Omega)} &\leq \eta_1 \label{thm:optimal_embedding_Lambda_spaces_not_SS:eq:5},\\
(1 + \varepsilon_1)\|u_1\chi_{B_1\setminus \tilde{B}_2}\|_{\Lambda^q_{w_{opt}}(\Omega)} & > C_2 \label{thm:optimal_embedding_Lambda_spaces_not_SS:eq:6},
\end{align}
and
\begin{equation}\label{thm:optimal_embedding_Lambda_spaces_not_SS:eq:7}
c_1^q W(\tilde{a}) < \frac1{4}. 
\end{equation}
Thanks to \eqref{thm:optimal_embedding_Lambda_spaces_not_SS:eq:6} combined with the dominated convergence theorem, we can find $0 < a_2 < \tilde{a}$ such that
\begin{equation}\label{thm:optimal_embedding_Lambda_spaces_not_SS:eq:17}
(1+\varepsilon_1)^q \int_{|B_2|}^{|B_1|} (u_1\chi_{B_1\setminus \tilde{B}_2})^*(t)^q w_{opt}(t) \dd{t} \geq C_2^q,
\end{equation}
where $B_2 = B(x_0, a_2^{1/n}R)$. Note that
\begin{equation*}
	a_2 < \frac{a_1}{4}
\end{equation*}
and that the inequalities \eqref{thm:optimal_embedding_Lambda_spaces_not_SS:eq:5}--\eqref{thm:optimal_embedding_Lambda_spaces_not_SS:eq:17} still hold with $\tilde{B}_2$ and $\tilde{a}$ replaced by $B_2$ and $a_2$, respectively. Having found $a_2$ and $B_2$, we now use \cref{thm:small_support_big_norm} again to find functions $f_2=c_2\chi_{(a_2/2,a_2)}\in\I_{a_2/2, a_2}$ for some $c_2>0$ and $u_2=u_{f_2}\in\VmXOm\cap\mathcal C^{m-1}(\Omega\setminus\{x_0\})$ such that \eqref{thm:optimal_embedding_Lambda_spaces_not_SS:eq:2}--\eqref{thm:optimal_embedding_Lambda_spaces_not_SS:eq:41} hold with $f_1$, $u_1$, and $B_1$ replaced by $f_2$, $u_2$, and $B_2$, respectively. 

We proceed inductively. More specifically, assume that for some $M\in\N$ we already have $\{a_j\}_{j = 1}^{M+1}$, $\{B_j\}_{j = 1}^{M+1}$, $\{f_j\}_{j=1}^M$, and $\{u_j\}_{j = 1}^M$ such that
\begin{align}
f_j &= c_j\chi_{(a_j/2,a_j)}\in\I_{a_j/2,a_j}, \label{thm:optimal_embedding_Lambda_spaces_not_SS:eq:8}\\
\|f_j\|_{\Lambda^q_w(0,1)} = c_j W\Big( \frac{a_j}{2} \Big)^\frac1{q} &= 1, \label{thm:optimal_embedding_Lambda_spaces_not_SS:eq:27}\\
u_j &= u_{f_j} \in \VmXOm\cap \mathcal{C}^{m-1}(\Omega\setminus\{x_0\}),\label{thm:optimal_embedding_Lambda_spaces_not_SS:eq:9}\\
\text{the support of $u_j$ is } &\text{the closed ball $\bar{B}_j$}, \label{thm:optimal_embedding_Lambda_spaces_not_SS:eq:13}\\
\|u_j\|_{\VmLamOm} &\leq C_1 \|f_j\|_{\Lambda^q_w(0,1)}, \label{thm:optimal_embedding_Lambda_spaces_not_SS:eq:25}\\
C_2 &\leq \|u_j\|_{\Lambda^q_{w_{opt}}(\Omega)} < \infty, \label{thm:optimal_embedding_Lambda_spaces_not_SS:eq:40}\\
\int_{|B_{j+1}|}^{|B_j|} (u_j\chi_{B_j\setminus B_{j+1}})^*(t)^q w_{opt}(t) \dd{t} &\geq \Big( \frac{C_2}{1 + \varepsilon_1} \Big)^q, \label{thm:optimal_embedding_Lambda_spaces_not_SS:eq:19}\\
\bar{B}_{j+1} &\subsetneq B_j, \label{thm:optimal_embedding_Lambda_spaces_not_SS:eq:10}\\
0 < a_{j+1} &< \frac{a_j}{4}, \label{thm:optimal_embedding_Lambda_spaces_not_SS:eq:11}\\
\|u_j\chi_{B_{j+1}}\|_{\Lambda^q_{w_{opt}}(\Omega)} &\leq \eta_j, \label{thm:optimal_embedding_Lambda_spaces_not_SS:eq:14}\\
(1 + \varepsilon_1)\|u_j\chi_{B_j\setminus B_{j+1}}\|_{\Lambda^q_{w_{opt}}(\Omega)} &> C_2, \label{thm:optimal_embedding_Lambda_spaces_not_SS:eq:15}\\
c_j^q W(a_{j+1}) &< \frac1{4^j}, \label{thm:optimal_embedding_Lambda_spaces_not_SS:eq:16}
\end{align}
for every $j=1,\dots, M$. Repeating the same argument as in the first step, we find functions $f_{M+1}=\alpha_{M+1}\chi_{(a_{M+1}/2,a_{M+1})}\in\I_{a_{M+1}/2,a_{M+1}}$, $u_{M+1}=u_{f_{M+1}}\in \VmXOm \cap \mathcal C^{m-1}(\Omega\setminus\{x_0\})$, a number $a_{M+2}$, and a ball $B_{M+2} = B(x_0, a_{M+2}^{1/n}R)$ such that \eqref{thm:optimal_embedding_Lambda_spaces_not_SS:eq:8}--\eqref{thm:optimal_embedding_Lambda_spaces_not_SS:eq:16} hold with $j$ replaced by $M+1$.

Now, it follows from \eqref{thm:optimal_embedding_Lambda_spaces_not_SS:eq:13} combined with \eqref{thm:optimal_embedding_Lambda_spaces_not_SS:eq:10} that the sequence $\{u_j\}_{j=1}^\infty$ is linear independent. We need to show that both \eqref{thm:optimal_embedding_Lambda_spaces_not_SS:optimal_norm_E} and \eqref{thm:optimal_embedding_Lambda_spaces_not_SS:gradient_norm_E} are satisfied. We start with the former. Let $\{\alpha_j\}_{j = 1}^\infty\in\ell_q$. For each $j\in\N$, set
\begin{equation*}
\tilde{u}_j = u_j \chi_{B_j\setminus B_{j+1}}.
\end{equation*}
Note that $\sum_{j=1}^\infty \alpha_j u_j$ is a well-defined measurable function because the sum is locally finite. Since the functions $\tilde{u}_j$ have mutually disjoint supports (recall \eqref{thm:optimal_embedding_Lambda_spaces_not_SS:eq:13} and \eqref{thm:optimal_embedding_Lambda_spaces_not_SS:eq:10}), we have
\begin{equation}\label{thm:optimal_embedding_Lambda_spaces_not_SS:eq:20}
\Big( \sum_{j=1}^\infty \alpha_j \tilde{u}_j \Big)^* \geq \sum_{j=1}^\infty |\alpha_j| \tilde{u}_j^*\chi_{(|B_{j+1}|, |B_j|)}.
\end{equation}
Clearly,
\begin{align}
\Big\| \sum_{j=1}^\infty \alpha_j u_j \Big\|_{\Lambda^q_{w_{opt}}(\Omega)} &\geq \Big\| \sum_{j=1}^\infty \alpha_j \tilde{u}_j \Big\|_{\Lambda^q_{w_{opt}}(\Omega)} - \Big\| \sum_{j=1}^\infty \alpha_j (u_j-\tilde{u}_j) \Big\|_{\Lambda^q_{w_{opt}}(\Omega)} \nonumber\\
&= \Big\| \sum_{j=1}^\infty \alpha_j \tilde{u}_j \Big\|_{\Lambda^q_{w_{opt}}(\Omega)} - \Big\| \sum_{j=1}^\infty \alpha_j u_j\chi_{B_{j+1}} \Big\|_{\Lambda^q_{w_{opt}}(\Omega)}. \label{thm:optimal_embedding_Lambda_spaces_not_SS:eq:21}
\end{align}
Moreover, note that the support of $\sum_{j=1}^\infty \alpha_j u_j$ is in the closed ball $\bar{B}_1$ thanks to \eqref{thm:optimal_embedding_Lambda_spaces_not_SS:eq:13} and \eqref{thm:optimal_embedding_Lambda_spaces_not_SS:eq:10}. As for the first term in \eqref{thm:optimal_embedding_Lambda_spaces_not_SS:eq:21}, using \eqref{thm:optimal_embedding_Lambda_spaces_not_SS:eq:20} and \eqref{thm:optimal_embedding_Lambda_spaces_not_SS:eq:19}, we have
\begin{align}
\Big\| \sum_{j=1}^\infty \alpha_j \tilde{u}_j \Big\|_{\Lambda^q_{w_{opt}}(\Omega)}^q &= \int_0^{|B_1|} \Big( \sum_{j=1}^\infty \alpha_j \tilde{u}_j \Big)^*(t)^q w_{opt}(t) \dd{t} \nonumber\\
&\geq\sum_{j=1}^\infty |\alpha_j|^q \int_{|B_{j+1}|}^{|B_j|} \tilde{u}_j^*(t)^q w_{opt}(t) \dd{t} \nonumber\\
&\geq\frac{C_2^q}{(1+\varepsilon_1)^q} \sum_{j=1}^\infty |\alpha_j|^q. \label{thm:optimal_embedding_Lambda_spaces_not_SS:eq:22}
\end{align}
As for the second term in \eqref{thm:optimal_embedding_Lambda_spaces_not_SS:eq:21}, using \eqref{thm:optimal_embedding_Lambda_spaces_not_SS:eq:14}, the H\"older inequality, and \eqref{thm:optimal_embedding_Lambda_spaces_not_SS:eq:23}, we see that
\begin{align}
\Big\| \sum_{j=1}^\infty \alpha_j u_j\chi_{B_{j+1}} \Big\|_{\Lambda^q_{w_{opt}}(\Omega)} &\leq \sum_{j=1}^\infty |\alpha_j| \|u_j\chi_{B_{j+1}}\|_{\Lambda^q_{w_{opt}}(\Omega)} \nonumber\\
&\leq \|\{\alpha_j\}_{j = 1}^\infty\|_{\ell_q} \|\{\eta_j\}_{j = 1}^\infty\|_{\ell_{q'}} \nonumber\\
&\leq \varepsilon_2 \|\{\alpha_j\}_{j = 1}^\infty\|_{\ell_q}. \label{thm:optimal_embedding_Lambda_spaces_not_SS:eq:24}
\end{align}
Finally, combining \eqref{thm:optimal_embedding_Lambda_spaces_not_SS:eq:21} with \eqref{thm:optimal_embedding_Lambda_spaces_not_SS:eq:22} and \eqref{thm:optimal_embedding_Lambda_spaces_not_SS:eq:24}, we obtain \eqref{thm:optimal_embedding_Lambda_spaces_not_SS:optimal_norm_E}.

It remains for us to show that \eqref{thm:optimal_embedding_Lambda_spaces_not_SS:gradient_norm_E} is valid. Let $N\in\N$ and $\{\alpha_j\}_{j=1}^N\subseteq\R$. Clearly, it is sufficient to show \eqref{thm:optimal_embedding_Lambda_spaces_not_SS:gradient_norm_E} for $\{\alpha_j\}_{j=1}^N\subseteq\R$ such that
\begin{equation}\label{thm:optimal_embedding_Lambda_spaces_not_SS:eq:28}
\sum_{j=1}^N |\alpha_j|^q = 1.
\end{equation}
Thanks to \eqref{thm:optimal_embedding_Lambda_spaces_not_SS:eq:9} and \eqref{thm:small_support_big_norm:E:controlled_Sobolev_norm_of_linear_comb}, we have
\begin{equation*}
\Big\| \sum_{j=1}^N \alpha_j u_j \Big\|_{\VmLamOm} \leq C_1 \Big\| \sum_{j=1}^N \alpha_j f_j \Big\|_{\Lambda^q_w(0,1)}.
\end{equation*}
Recall that $f_j = c_j\chi_{(a_j/2,a_j)}$ for some $c_j>0$ (see~\eqref{thm:optimal_embedding_Lambda_spaces_not_SS:eq:8}). Note that it follows from \eqref{thm:optimal_embedding_Lambda_spaces_not_SS:eq:11} that the sequence $\{a_j\}_{j = 1}^\infty$ is decreasing and the intervals $\{(a_j/2,a_j)\}_{j = 1}^\infty$ are mutually disjoint. Furthermore, since the sequence $\{a_j\}_{j = 1}^\infty$ is decreasing, it follows from \eqref{thm:optimal_embedding_Lambda_spaces_not_SS:eq:27} that the sequence $\{c_j\}_{j=1}^\infty$ is nondecreasing. Note that
\begin{equation}\label{thm:optimal_embedding_Lambda_spaces_not_SS:eq:29}
\sum_{k=l}^\infty \frac{a_k}{2} \leq \frac{a_l}{2} \sum_{k = l}^\infty 4^{l - k} = \frac{2}{3}a_l <  a_l  \quad \text{for every $l\in\N$}
\end{equation}
thanks to \eqref{thm:optimal_embedding_Lambda_spaces_not_SS:eq:11}. Furthermore, note that the support of $\sum_{j=1}^N \alpha_j f_j$ is in the interval $(0,a_1]$. Hence, using \eqref{thm:optimal_embedding_Lambda_spaces_not_SS:eq:29} with $l = 1$ and \eqref{prel:pointwise_inequality_sum_of_rearrangements}, we obtain
\begin{align}
\Big\| \sum_{j=1}^N \alpha_j f_j \Big\|_{\Lambda^q_w(0,1)}^q &= \int_0^{a_1} \Big( \sum_{j=1}^N \alpha_j f_j \Big)^*(t)^q w(t) \dd{t} \nonumber\\
&= \int_0^{2a_1} \Big( \sum_{j=1}^N \alpha_j f_j \Big)^*(t)^q w(t) \dd{t} \nonumber\\
&= \sum_{k=1}^\infty \int_{2a_{k+1}}^{2a_k} \Big( \sum_{j=1}^N \alpha_j f_j \Big)^*(t)^q w(t) \dd{t} \nonumber\\
&= \sum_{k=1}^N \int_{2a_{k+1}}^{2a_k} \Big( \sum_{j=1}^N \alpha_j f_j \Big)^*(t)^q w(t) \dd{t} \nonumber\\
&\quad+ \sum_{k=N+1}^\infty \int_{2a_{k+1}}^{2a_k} \Big( \sum_{j=1}^N \alpha_j f_j \Big)^*(t)^q w(t) \dd{t}\nonumber\\
&\leq \sum_{k=1}^N \int_{2a_{k+1}}^{2a_k} \Big( \sum_{j=1}^{k} \alpha_j f_j \Big)^*(t/2)^q w(t) \dd{t} \nonumber\\
&\quad+ \sum_{k=1}^{N-1} \int_{2a_{k+1}}^{2a_k} \Big( \sum_{j=k+1}^N \alpha_j f_j \Big)^*(t/2)^q w(t) \dd{t} \nonumber\\
&\quad+ \sum_{k=N+1}^\infty \int_{2a_{k+1}}^{2a_k} \Big( \sum_{j=1}^N \alpha_j f_j \Big)^*(t)^q w(t) \dd{t} \label{thm:optimal_embedding_Lambda_spaces_not_SS:eq:31}.
\end{align}

Now, as for the first term, note that
\begin{equation*}
\Big| \sum_{j=1}^k \alpha_j f_j \Big| = \sum_{j=1}^k |\alpha_j| f_j \leq |\alpha_k| f_k + \sum_{j=1}^{k-1} f_j \quad \text{for every $k=1,\dots, N$}
\end{equation*}
thanks to \eqref{thm:optimal_embedding_Lambda_spaces_not_SS:eq:28}. Using this, \eqref{prel:pointwise_inequality_sum_of_rearrangements}, and \eqref{thm:optimal_embedding_Lambda_spaces_not_SS:eq:8}, we obtain
\begin{align}
\sum_{k=1}^N \int_{2a_{k+1}}^{2a_k} \Big( \sum_{j=1}^{k} \alpha_j f_j \Big)^*(t/2)^q w(t) \dd{t} &\leq \sum_{k=1}^N \int_{2a_{k+1}}^{2a_k} \Big( |\alpha_k| f_k + \sum_{j=1}^{k-1} f_j \Big)^*(t/2)^q w(t) \dd{t} \nonumber\\
&\begin{aligned}\leq 2^{q-1} \sum_{k=1}^N \Bigg( &|\alpha_k|^q |c_k|^q \int_{2a_{k+1}}^{2a_k} \chi_{(0, a_k/2)}(t/4) w(t) \dd{t}\\&+ \int_{2a_{k+1}}^{2a_k} \Big( \sum_{j=1}^{k-1} f_j \Big)^*(t/4)^q w(t) \dd{t} \Bigg)\end{aligned} \nonumber\\
&\begin{aligned}\leq 2^{q-1} \sum_{k=1}^N\Bigg( &|\alpha_k|^q c_k^q W(2a_k) \\+ &\int_{2a_{k+1}}^{2a_k} \Big(\sum_{j=1}^{k-1} f_j \Big)^*(t/4)^q w(t) \dd{t} \Bigg).\end{aligned} \label{thm:optimal_embedding_Lambda_spaces_not_SS:eq:38}
\end{align}
Furthermore, we have
\begin{align}
\sum_{k=1}^N |\alpha_k|^q c_k^q W(2a_k) &\leq \Delta_2^2 \sum_{k=1}^N |\alpha_k|^q c_k^q W(a_k/2) \nonumber\\
&= \Delta_2^2 \sum_{k=1}^N |\alpha_k|^q \label{thm:optimal_embedding_Lambda_spaces_not_SS:eq:32}
\end{align}
thanks to \eqref{thm:optimal_embedding_Lambda_spaces_not_SS:eq:30} and \eqref{thm:optimal_embedding_Lambda_spaces_not_SS:eq:27}. Using the disjointness of the intervals $\{(a_j/2,a_j)\}_{j=1}^\infty$ and the monotonicity of $\{c_j\}_{j=1}^\infty$, it is easy to see that
\begin{equation}\label{thm:optimal_embedding_Lambda_spaces_not_SS:eq:34}
\Big( \sum_{j=i}^M f_j \Big)^* = \sum_{j=i}^M c_j \chi_{\big( \sum_{l=j+1}^M a_l/2, \sum_{l=j}^M a_l/2 \big)} \quad \text{for every $1\leq i \leq M \in\N$}.
\end{equation}
In particular, we have
\begin{equation}\label{thm:optimal_embedding_Lambda_spaces_not_SS:eq:43}
\Big( \sum_{j=i}^M f_j \Big)^* \leq c_M \quad \text{for every $1\leq i \leq M \in\N$}.
\end{equation}
Using \eqref{thm:optimal_embedding_Lambda_spaces_not_SS:eq:43}, \eqref{thm:optimal_embedding_Lambda_spaces_not_SS:eq:30}, and \eqref{thm:optimal_embedding_Lambda_spaces_not_SS:eq:16}, we obtain, for all $k\in\N$, $k\geq2$, and $\gamma>0$,
\begin{align}
\int_{2a_{k+1}}^{2a_k} \Big(\sum_{j=1}^{k-1} f_j \Big)^*(t/\gamma)^q w(t) \dd{t} &\leq \sum_{k=1}^N c_{k-1}^q W(2a_k) \nonumber\\
&\leq \Delta_2 c_{k-1}^q W(a_k) \nonumber\\
&\leq \frac{\Delta_2}{4^{k-1}}. \label{thm:optimal_embedding_Lambda_spaces_not_SS:eq:35}
\end{align}
Hence, combining \eqref{thm:optimal_embedding_Lambda_spaces_not_SS:eq:38}, \eqref{thm:optimal_embedding_Lambda_spaces_not_SS:eq:32}, and \eqref{thm:optimal_embedding_Lambda_spaces_not_SS:eq:35}, we arrive at
\begin{equation}\label{thm:optimal_embedding_Lambda_spaces_not_SS:eq:39}
\sum_{k=1}^N \int_{a_{k+1}}^{a_k} \Big( \sum_{j=1}^{k} \alpha_j f_j \Big)^*(t/2)^q w(t) \dd{t} \leq 2^{q-1}\Big( \Delta_2^2 \sum_{k=1}^N |\alpha_k|^q + \Delta_2 \sum_{k=2}^N \frac1{4^{k-1}} \Big).
\end{equation}

As for the second term in \eqref{thm:optimal_embedding_Lambda_spaces_not_SS:eq:31}, fix $k=1, \dots, N-1$. Using \eqref{thm:optimal_embedding_Lambda_spaces_not_SS:eq:28} and \eqref{thm:optimal_embedding_Lambda_spaces_not_SS:eq:34} together with the disjointness of the intervals $\{(a_j/2,a_j)\}_{j=1}^\infty$ once more, we obtain
\begin{align}
\int_{2a_{k+1}}^{2a_k} \Big( \sum_{j=k+1}^N \alpha_j f_j \Big)^*(t/2)^q w(t) \dd{t} &\leq \int_{2a_{k+1}}^{2a_k} \Big( \sum_{j=k+1}^N f_j \Big)^*(t/2)^q w(t) \dd{t} \nonumber\\
&\leq \int_{2a_{k+1}}^{2a_k} \sum_{j=k+1}^N c_j^q \chi_{\big( \sum_{l=j+1}^N a_l, \sum_{l=j}^N a_l \big)}(t) w(t) \dd{t} \nonumber\\
&\leq \sum_{j=k+1}^N c_j^q \int_{2a_{k+1}}^{2a_k}   \chi_{(0, \sum_{l=j}^N a_l)}(t) w(t) \dd{t}. \label{thm:optimal_embedding_Lambda_spaces_not_SS:eq:12}
\end{align}
Note that
\begin{equation}\label{thm:optimal_embedding_Lambda_spaces_not_SS:eq:42}
\sum_{l=j}^N a_l < 2a_j \leq 2a_{k+1} \quad \text{for every $k+1\leq j\leq N$}
\end{equation}
thanks to \eqref{thm:optimal_embedding_Lambda_spaces_not_SS:eq:29} and the monotonicity of $\{a_j\}_{j = 1}^\infty$. Therefore, in view of \eqref{thm:optimal_embedding_Lambda_spaces_not_SS:eq:12} and \eqref{thm:optimal_embedding_Lambda_spaces_not_SS:eq:42}, we have
\begin{equation*}
\int_{2a_{k+1}}^{2a_k} \Big( \sum_{j=k+1}^N \alpha_j f_j \Big)^*(t/2)^q w(t) \dd{t} = 0,
\end{equation*}
whence it follows that
\begin{equation}\label{thm:optimal_embedding_Lambda_spaces_not_SS:eq:36}
\sum_{k=1}^{N-1} \int_{2a_{k+1}}^{2a_k} \Big( \sum_{j=k+1}^N \alpha_j f_j \Big)^*(t/2)^q w(t) \dd{t} = 0.
\end{equation}
Finally, as for the third term in \eqref{thm:optimal_embedding_Lambda_spaces_not_SS:eq:31}, using \eqref{thm:optimal_embedding_Lambda_spaces_not_SS:eq:28} and \eqref{thm:optimal_embedding_Lambda_spaces_not_SS:eq:35}, we obtain
\begin{align}
\sum_{k=N+1}^\infty \int_{2a_{k+1}}^{2a_k} \Big( \sum_{j=1}^N \alpha_j f_j \Big)^*(t)^q w(t) \dd{t} &\leq \sum_{k=N+1}^\infty \int_{2a_{k+1}}^{2a_k} \Big( \sum_{j=1}^N f_j \Big)^*(t)^q w(t) \dd{t} \nonumber\\
&\leq \Delta_2 \sum_{k=N+1}^\infty 4^{-k + 1}. \label{thm:optimal_embedding_Lambda_spaces_not_SS:eq:37}
\end{align}

At last, combining \eqref{thm:optimal_embedding_Lambda_spaces_not_SS:eq:39}, \eqref{thm:optimal_embedding_Lambda_spaces_not_SS:eq:36}, and \eqref{thm:optimal_embedding_Lambda_spaces_not_SS:eq:37} with \eqref{thm:optimal_embedding_Lambda_spaces_not_SS:eq:31}, we obtain
\begin{align*}
\Big\| \sum_{j=1}^N \alpha_j f_j \Big\|_{\Lambda^q_w(\Omega)}^q &\leq 2^{q-1} \Big(\Delta_2^2 \sum_{k=1}^N |\alpha_k|^q + \Delta_2\sum_{k=2}^\infty 4^{-k + 1} \Big) \\
&\leq 2^{q-1}\Delta_2^2 \Big(1 + \sum_{k=1}^N |\alpha_k|^q \Big).
\end{align*}
In view of \eqref{thm:optimal_embedding_Lambda_spaces_not_SS:eq:28}, this is \eqref{thm:optimal_embedding_Lambda_spaces_not_SS:gradient_norm_E} with $C_3 = 2\Delta_2^{2/q}$, which concludes the proof.
\end{proof}

\section*{Acknowledgment}
The authors would like to thank the referee for carefully reading the paper and their valuable comments.


\begin{thebibliography}{59}
\providecommand{\natexlab}[1]{#1}
\providecommand{\url}[1]{\texttt{#1}}
\expandafter\ifx\csname urlstyle\endcsname\relax
  \providecommand{\doi}[1]{doi: #1}\else
  \providecommand{\doi}{doi: \begingroup \urlstyle{rm}\Url}\fi

\bibitem[Adams and Fournier(2003)]{AFbook}
R.A. Adams and J.J.F. Fournier.
\newblock \emph{Sobolev spaces}, volume 140 of \emph{Pure and Applied
  Mathematics (Amsterdam)}.
\newblock Elsevier/Academic Press, Amsterdam, second edition, 2003.

\bibitem[Albiac and Kalton(2016)]{AK:16}
F.~Albiac and N.~J. Kalton.
\newblock \emph{Topics in {B}anach space theory}, volume 233 of \emph{Graduate
  Texts in Mathematics}.
\newblock Springer, [Cham], second edition, 2016.
\newblock \doi{10.1007/978-3-319-31557-7}.

\bibitem[Baernstein(2019)]{B:19}
A.~Baernstein, II.
\newblock \emph{Symmetrization in analysis}, volume~36 of \emph{New
  Mathematical Monographs}.
\newblock Cambridge University Press, Cambridge, 2019.
\newblock \doi{10.1017/9781139020244}.

\bibitem[Bennett and Rudnick(1980)]{BR:80}
C.~Bennett and K.~Rudnick.
\newblock On {L}orentz-{Z}ygmund spaces.
\newblock \emph{Dissertationes Math. (Rozprawy Mat.)}, 175:\penalty0 67, 1980.

\bibitem[Bennett and Sharpley(1988)]{BS}
C.~Bennett and R.~Sharpley.
\newblock \emph{{Interpolation of operators}}, volume 129 of \emph{Pure and
  Applied Mathematics}.
\newblock Academic Press, Inc., Boston, MA, 1988.

\bibitem[Breit and Cianchi(2021)]{BC:21}
D.~Breit and A.~Cianchi.
\newblock Symmetric gradient {S}obolev spaces endowed with
  rearrangement-invariant norms.
\newblock \emph{Adv. Math.}, 391:\penalty0 107954, 101 pages, 2021.
\newblock \doi{10.1016/j.aim.2021.107954}.

\bibitem[Brezis and Wainger(1980)]{BW:80}
H.~Brezis and S.~Wainger.
\newblock A note on limiting cases of {S}obolev embeddings and convolution
  inequalities.
\newblock \emph{Comm. Partial Differential Equations}, 5\penalty0 (7):\penalty0
  773--789, 1980.
\newblock \doi{10.1080/03605308008820154}.

\bibitem[Carl and Stephani(1990)]{CS:90}
B.~Carl and I.~Stephani.
\newblock \emph{Entropy, compactness and the approximation of operators},
  volume~98 of \emph{Cambridge Tracts in Mathematics}.
\newblock Cambridge University Press, Cambridge, 1990.
\newblock \doi{10.1017/CBO9780511897467}.

\bibitem[Carro et~al.(1996)Carro, Garc\'{\i}a~del Amo, and Soria]{CGS:96}
M.J. Carro, A.~Garc\'{\i}a~del Amo, and J.~Soria.
\newblock Weak-type weights and normable {L}orentz spaces.
\newblock \emph{Proc. Amer. Math. Soc.}, 124\penalty0 (3):\penalty0 849--857,
  1996.
\newblock \doi{10.1090/S0002-9939-96-03214-5}.

\bibitem[Cavaliere and Mihula(2019)]{CM:19}
P.~Cavaliere and Z.~Mihula.
\newblock Compactness for {S}obolev-type trace operators.
\newblock \emph{Nonlinear Anal.}, 183:\penalty0 42--69, 2019.
\newblock \doi{10.1016/j.na.2019.01.013}.

\bibitem[Cavaliere and Mihula(2022)]{CM:22}
P.~Cavaliere and Z.~Mihula.
\newblock Compactness of {S}obolev-type embeddings with measures.
\newblock \emph{Commun. Contemp. Math.}, 24\penalty0 (9):\penalty0 Paper No.
  2150036, 41 pages, 2022.
\newblock \doi{10.1142/S021919972150036X}.

\bibitem[Chuah et~al.()Chuah, Lang, and Yao]{CLY:25preprint}
C.Y. Chuah, J.~Lang, and L.~Yao.
\newblock A {B}ourgain-{G}romov problem on non-compact {S}obolev-{L}orentz
  embeddings.
\newblock Preprint.\ arXiv:2502.05308 [math.FA].

\bibitem[Cianchi and Pick(2016)]{CP:16}
A.~Cianchi and L.~Pick.
\newblock Optimal {S}obolev trace embeddings.
\newblock \emph{Trans. Amer. Math. Soc.}, 368\penalty0 (12):\penalty0
  8349--8382, 2016.
\newblock \doi{10.1090/tran/6606}.

\bibitem[Cianchi et~al.(2015)Cianchi, Pick, and Slav\'{\i}kov\'{a}]{CPS:15}
A.~Cianchi, L.~Pick, and L.~Slav\'{\i}kov\'{a}.
\newblock Higher-order {S}obolev embeddings and isoperimetric inequalities.
\newblock \emph{Adv. Math.}, 273:\penalty0 568--650, 2015.
\newblock \doi{10.1016/j.aim.2014.12.027}.

\bibitem[Cianchi et~al.(2020)Cianchi, Pick, and Slav\'{\i}kov\'{a}]{CPS:20}
A.~Cianchi, L.~Pick, and L.~Slav\'{\i}kov\'{a}.
\newblock Sobolev embeddings, rearrangement-invariant spaces and {F}rostman
  measures.
\newblock \emph{Ann. Inst. H. Poincar\'{e} C Anal. Non Lin\'{e}aire},
  37\penalty0 (1):\penalty0 105--144, 2020.
\newblock \doi{10.1016/j.anihpc.2019.06.004}.

\bibitem[Donaldson(1971)]{D:71}
T.~Donaldson.
\newblock Nonlinear elliptic boundary value problems in {O}rlicz-{S}obolev
  spaces.
\newblock \emph{J. Differential Equations}, 10:\penalty0 507--528, 1971.

\bibitem[Edmunds and Triebel(1996)]{ET:96}
D.E. Edmunds and H.~Triebel.
\newblock \emph{Function spaces, entropy numbers, differential operators},
  volume 120 of \emph{Cambridge Tracts in Mathematics}.
\newblock Cambridge University Press, Cambridge, 1996.
\newblock \doi{10.1017/CBO9780511662201}.

\bibitem[Edmunds et~al.(2000)Edmunds, Kerman, and Pick]{EKP:00}
D.E. Edmunds, R.~Kerman, and L.~Pick.
\newblock Optimal {S}obolev imbeddings involving rearrangement-invariant
  quasinorms.
\newblock \emph{J. Funct. Anal.}, 170\penalty0 (2):\penalty0 307--355, 2000.
\newblock \doi{10.1006/jfan.1999.3508}.

\bibitem[Edmunds et~al.(2020)Edmunds, Mihula, Musil, and Pick]{EMMP:20}
D.E. Edmunds, Z.~Mihula, V.~Musil, and L.~Pick.
\newblock Boundedness of classical operators on rearrangement-invariant spaces.
\newblock \emph{J. Funct. Anal.}, 278\penalty0 (4):\penalty0 108341, 56 pages,
  2020.
\newblock \doi{10.1016/j.jfa.2019.108341}.

\bibitem[Gagliardo(1958)]{G:58}
E.~Gagliardo.
\newblock Propriet\`a di alcune classi di funzioni in pi\`u variabili.
\newblock \emph{Ricerche Mat.}, 7:\penalty0 102--137, 1958.

\bibitem[Gogatishvili and Soudsk\'y(2014)]{GS:14}
A.~Gogatishvili and F.~Soudsk\'y.
\newblock Normability of {L}orentz spaces---an alternative approach.
\newblock \emph{Czechoslovak Math. J.}, 64(139)\penalty0 (3):\penalty0
  581--597, 2014.
\newblock \doi{10.1007/s10587-014-0120-y}.

\bibitem[Gogatishvili et~al.(2006)Gogatishvili, Opic, and Pick]{GOP:06}
A.~Gogatishvili, B.~Opic, and L.~Pick.
\newblock {Weighted inequalities for Hardy-type operators involving suprema}.
\newblock \emph{Collect. Math.}, 57\penalty0 (3):\penalty0 227--255, 2006.

\bibitem[Gurka et~al.(2025)Gurka, Lang, and Mihula]{GLM:25}
P.~Gurka, J.~Lang, and Z.~Mihula.
\newblock Quantitative analysis of optimal {S}obolev-{L}orentz embeddings with
  {$\alpha$}-homogeneous weights.
\newblock \emph{J. Geom. Anal.}, 35\penalty0 (5):\penalty0 Paper No. 146, 22
  pages, 2025.
\newblock \doi{10.1007/s12220-025-01946-0}.

\bibitem[Hansson(1979)]{H:79}
K.~Hansson.
\newblock Imbedding theorems of {S}obolev type in potential theory.
\newblock \emph{Math. Scand.}, 45\penalty0 (1):\penalty0 77--102, 1979.
\newblock \doi{10.7146/math.scand.a-11827}.

\bibitem[Kami\'nska and Maligranda(2004)]{KM:04}
A.~Kami\'nska and L.~Maligranda.
\newblock Order convexity and concavity of {L}orentz spaces {$\Lambda_{p,w},\
  0<p<\infty$}.
\newblock \emph{Studia Math.}, 160\penalty0 (3):\penalty0 267--286, 2004.
\newblock \doi{10.4064/sm160-3-5}.

\bibitem[Kerman and Pick(2006)]{KP:06}
R.~Kerman and L.~Pick.
\newblock Optimal {S}obolev imbeddings.
\newblock \emph{Forum Math.}, 18\penalty0 (4):\penalty0 535--570, 2006.
\newblock \doi{10.1515/FORUM.2006.028}.

\bibitem[Kerman and Pick(2008)]{KP:08}
R.~Kerman and L.~Pick.
\newblock Compactness of {S}obolev imbeddings involving rearrangement-invariant
  norms.
\newblock \emph{Studia Math.}, 186\penalty0 (2):\penalty0 127--160, 2008.
\newblock \doi{10.4064/sm186-2-2}.

\bibitem[Lang and Mihula(2023)]{LM:23}
J.~Lang and Z.~Mihula.
\newblock Different degrees of non-compactness for optimal {S}obolev
  embeddings.
\newblock \emph{J. Funct. Anal.}, 284\penalty0 (10):\penalty0 Paper No. 109880,
  22, 2023.
\newblock \doi{10.1016/j.jfa.2023.109880}.

\bibitem[Lang and Musil(2019)]{LM:19}
J.~Lang and V.~Musil.
\newblock Strict {$s$}-numbers of non-compact {S}obolev embeddings into
  continuous functions.
\newblock \emph{Constr. Approx.}, 50\penalty0 (2):\penalty0 271--291, 2019.
\newblock \doi{10.1007/s00365-018-9448-0}.

\bibitem[Lang et~al.(2021)Lang, Musil, Ol\v{s}\'{a}k, and Pick]{LMOP:21}
J.~Lang, V.~Musil, M.~Ol\v{s}\'{a}k, and L.~Pick.
\newblock Maximal non-compactness of {S}obolev embeddings.
\newblock \emph{J. Geom. Anal.}, 31\penalty0 (9):\penalty0 9406--9431, 2021.
\newblock \doi{10.1007/s12220-020-00522-y}.

\bibitem[Lang et~al.(2024)Lang, Mihula, and Pick]{LMP:24FirstView}
J.~Lang, Z.~Mihula, and L.~Pick.
\newblock Maximal noncompactness of limiting {S}obolev embeddings.
\newblock \emph{Proc. Roy. Soc. Edinburgh Sect. A}, First View:\penalty0 19
  pages, 2024.
\newblock \doi{10.1017/prm.2024.93}.

\bibitem[Lef\`evre and Rodr\'{\i}guez-Piazza(2014)]{LR-P:14}
P.~Lef\`evre and L.~Rodr\'{\i}guez-Piazza.
\newblock Finitely strictly singular operators in harmonic analysis and
  function theory.
\newblock \emph{Adv. Math.}, 255:\penalty0 119--152, 2014.
\newblock \doi{10.1016/j.aim.2013.12.034}.

\bibitem[Lieb and Loss(2001)]{LL:01}
E.H. Lieb and M.~Loss.
\newblock \emph{Analysis}, volume~14 of \emph{Graduate Studies in Mathematics}.
\newblock American Mathematical Society, Providence, RI, second edition, 2001.
\newblock \doi{10.1090/gsm/014}.

\bibitem[Maz'ya(1973)]{M:73}
V.G. Maz'ya.
\newblock On certain integral inequalities for functions of many variables.
\newblock \emph{J. Sov. Math.}, 1\penalty0 (2):\penalty0 205--234, 1973.
\newblock \doi{10.1007/BF01083775}.

\bibitem[Maz'ya(2011)]{Mabook}
V.G. Maz'ya.
\newblock \emph{Sobolev spaces with applications to elliptic partial
  differential equations}, volume 342 of \emph{Grundlehren der Mathematischen
  Wissenschaften [Fundamental Principles of Mathematical Sciences]}.
\newblock Springer, Heidelberg, augmented edition, 2011.
\newblock \doi{10.1007/978-3-642-15564-2}.

\bibitem[Mihula()]{M:25preprint}
Z.~Mihula.
\newblock Compact {S}obolev embeddings of radially symmetric functions.
\newblock Preprint.\ arXiv:2503.05922 [math.FA].

\bibitem[Musil and O\soft{l}hava(2019)]{MO:19}
V.~Musil and R.~O\soft{l}hava.
\newblock Interpolation theorem for {M}arcinkiewicz spaces with applications to
  {L}orentz gamma spaces.
\newblock \emph{Math. Nachr.}, 292\penalty0 (5):\penalty0 1106--1121, 2019.
\newblock \doi{10.1002/mana.201700452}.

\bibitem[Musil et~al.(2023)Musil, Pick, and Tak\'{a}\v{c}]{MPT:23}
V.~Musil, L.~Pick, and J.~Tak\'{a}\v{c}.
\newblock Optimality problems in {O}rlicz spaces.
\newblock \emph{Adv. Math.}, 432:\penalty0 Paper No. 109273, 58 pages, 2023.
\newblock \doi{10.1016/j.aim.2023.109273}.

\bibitem[Nirenberg(1959)]{N:59}
L.~Nirenberg.
\newblock On elliptic partial differential equations.
\newblock \emph{Ann. Scuola Norm. Sup. Pisa (3)}, 13:\penalty0 115--162, 1959.

\bibitem[O'Neil(1963)]{On:63}
R.~O'Neil.
\newblock Convolution operators and {$L(p,\,q)$} spaces.
\newblock \emph{Duke Math. J.}, 30:\penalty0 129--142, 1963.
\newblock \doi{10.1215/S0012-7094-63-03015-1}.

\bibitem[Opic and Pick(1999)]{OP:99}
B.~Opic and L.~Pick.
\newblock On generalized {L}orentz-{Z}ygmund spaces.
\newblock \emph{Math. Inequal. Appl.}, 2\penalty0 (3):\penalty0 391--467, 1999.
\newblock \doi{10.7153/mia-02-35}.

\bibitem[Peetre(1966)]{P:66}
J.~Peetre.
\newblock Espaces d'interpolation et th\'{e}or\`eme de {S}oboleff.
\newblock \emph{Ann. Inst. Fourier (Grenoble)}, 16\penalty0 (1):\penalty0
  279--317, 1966.
\newblock \doi{10.5802/aif.232}.

\bibitem[Pinkus(1985)]{P:85}
A.~Pinkus.
\newblock \emph{{$n$}-widths in approximation theory}, volume~7 of
  \emph{Ergebnisse der Mathematik und ihrer Grenzgebiete (3) [Results in
  Mathematics and Related Areas (3)]}.
\newblock Springer-Verlag, Berlin, 1985.
\newblock \doi{10.1007/978-3-642-69894-1}.

\bibitem[Plichko(2004)]{P:04}
A.~Plichko.
\newblock Superstrictly singular and superstrictly cosingular operators.
\newblock In \emph{Functional analysis and its applications}, volume 197 of
  \emph{North-Holland Math. Stud.}, pages 239--255. Elsevier Sci. B. V.,
  Amsterdam, 2004.
\newblock \doi{10.1016/S0304-0208(04)80172-5}.

\bibitem[Pohozhaev(1965)]{P:65}
S.I. Pohozhaev.
\newblock On the imbedding sobolev theorem for $pl = n$.
\newblock In \emph{Doklady Conference, Section Math. Moscow Power Inst}, volume
  165, pages 158--170, 1965.

\bibitem[Sawyer(1990)]{S:90}
E.~Sawyer.
\newblock Boundedness of classical operators on classical {L}orentz spaces.
\newblock \emph{Studia Math.}, 96\penalty0 (2):\penalty0 145--158, 1990.

\bibitem[Sharpley(1972)]{S:72}
R.~Sharpley.
\newblock Spaces {$\Lambda_{\alpha }(X)$} and interpolation.
\newblock \emph{J. Functional Analysis}, 11:\penalty0 479--513, 1972.
\newblock \doi{10.1016/0022-1236(72)90068-7}.

\bibitem[Slav\'{\i}kov\'{a}(2012)]{S:12}
L.~Slav\'{\i}kov\'{a}.
\newblock Almost-compact embeddings.
\newblock \emph{Math. Nachr.}, 285\penalty0 (11-12):\penalty0 1500--1516, 2012.
\newblock \doi{10.1002/mana.201100286}.

\bibitem[Slav\'{\i}kov\'{a}(2015)]{S:15}
L.~Slav\'{\i}kov\'{a}.
\newblock Compactness of higher-order {S}obolev embeddings.
\newblock \emph{Publ. Mat.}, 59\penalty0 (2):\penalty0 373--448, 2015.
\newblock \doi{10.5565/PUBLMAT_59215_06}.

\bibitem[Sobolev(1938)]{Sob38}
S.L. Sobolev.
\newblock On a theorem of functional analysis.
\newblock \emph{Mat. Sb.}, 46:\penalty0 471--496, 1938.

\bibitem[Stein(1981)]{S:81}
E.M. Stein.
\newblock Editor's note: the differentiability of functions in {${\bf R}^{n}$}.
\newblock \emph{Ann. of Math. (2)}, 113\penalty0 (2):\penalty0 383--385, 1981.

\bibitem[Stepanov(1993)]{S:93}
V.D. Stepanov.
\newblock The weighted {H}ardy's inequality for nonincreasing functions.
\newblock \emph{Trans. Amer. Math. Soc.}, 338\penalty0 (1):\penalty0 173--186,
  1993.
\newblock \doi{10.2307/2154450}.

\bibitem[Strichartz(1972)]{S:72b}
R.S. Strichartz.
\newblock A note on {T}rudinger's extension of {S}obolev's inequalities.
\newblock \emph{Indiana Univ. Math. J.}, 21:\penalty0 841--842, 1972.
\newblock \doi{10.1512/iumj.1972.21.21066}.

\bibitem[Talenti(1976{\natexlab{a}})]{T:76}
G.~Talenti.
\newblock Best constant in {S}obolev inequality.
\newblock \emph{Ann. Mat. Pura Appl. (4)}, 110:\penalty0 353--372,
  1976{\natexlab{a}}.
\newblock \doi{10.1007/BF02418013}.

\bibitem[Talenti(1976{\natexlab{b}})]{T:76b}
G.~Talenti.
\newblock Elliptic equations and rearrangements.
\newblock \emph{Ann. Scuola Norm. Sup. Pisa Cl. Sci. (4)}, 3\penalty0
  (4):\penalty0 697--718, 1976{\natexlab{b}}.

\bibitem[Tartar(1998)]{T:98}
L.~Tartar.
\newblock Imbedding theorems of {S}obolev spaces into {L}orentz spaces.
\newblock \emph{Boll. Unione Mat. Ital. Sez. B Artic. Ric. Mat. (8)},
  1\penalty0 (3):\penalty0 479--500, 1998.

\bibitem[Triebel(1983)]{T:83}
H.~Triebel.
\newblock \emph{Theory of function spaces}, volume~78 of \emph{Monographs in
  Mathematics}.
\newblock Birkh\"auser Verlag, Basel, 1983.
\newblock \doi{10.1007/978-3-0346-0416-1}.

\bibitem[Trudinger(1967)]{T:67}
N.S. Trudinger.
\newblock On imbeddings into {O}rlicz spaces and some applications.
\newblock \emph{J. Math. Mech.}, 17:\penalty0 473--483, 1967.
\newblock \doi{10.1512/iumj.1968.17.17028}.

\bibitem[Yudovich(1961)]{Y:61}
V.I. Yudovich.
\newblock Some estimates connected with integral operators and with solutions
  of elliptic equations.
\newblock \emph{Dokl. Akad. Nauk SSSR}, 138:\penalty0 805--808, 1961.

\end{thebibliography}
\end{document}